\documentclass[journal]{IEEEtran}
\usepackage[utf8]{inputenc}
\usepackage[numbers]{natbib}
\usepackage{url}
\usepackage{physics} 
\usepackage{graphicx}
\usepackage{amsmath}
\usepackage[ruled,vlined]{algorithm2e}
\usepackage{amssymb,setspace,amsthm,calrsfs}
\usepackage{pgfplotstable,pstool,psfrag}
\pgfplotsset{compat=1.15}
\usepackage{float}
\usepackage{tabularx}
\usepackage{ltablex}
\usepackage{booktabs}
\usepackage{picins}
\usepackage{tabularx}
\usepackage{placeins}
\usepackage{supertabular}
\usepackage{tikz-cd}
\usepackage{xcolor}
\usepackage{color}
\definecolor{myblue}{RGB}{0,128,128}
\usepackage{acronym}

\newcommand{\R}{\mathbb{R}}
\newcommand{\N}{\mathbb{N}}

\newcommand{\probability}{\mathbb{P}}

\def\defeq{\mathrel{\mathop:}=}
\newtheorem{thm}{Theorem}
\newtheorem{dfn}{Definition}

\makeatletter
\let\mytagform@=\tagform@
\def\tagform@#1{\maketag@@@{\color{myblue}(#1)}}
\newif\ifnoindentafter
\noindentafterfalse
\makeatother

\newif\iffinalfonts
\iffinalfonts
\else

\fi

\SetAlFnt{\small}
\SetAlCapFnt{\small}

\usepackage[urlcolor=myblue, citecolor=myblue, colorlinks=true,linkcolor=myblue,urlcolor=myblue,citecolor=myblue]{hyperref}
\hypersetup{linkbordercolor=myblue}
\SetAlCapNameFnt{\small}
\SetAlCapHSkip{0pt}
\IncMargin{-\parindent}
\begin{document}

\title{Predictability and Fairness in Load Aggregation \\ and Operations of Virtual Power Plants}
 \author{
 J. Mare\v{c}ek, 
 M. Roubal{\' i}k, 
 R. Ghosh,
 R. Shorten,  
 F. Wirth
\thanks{J. Marecek and M. Roubalik are at the Czech Technical University, Prague, the Czech Republic.}
\thanks{R. Ghosh and R. Shorten were at University College Dublin, Dublin, Ireland.}
\thanks{R. Shorten is at Imperial College London, South Kensington, UK.}
\thanks{F. Wirth is at University of Passau, Germany.}
\thanks{Manuscript received \today.}
}
\maketitle
\begin{abstract}
In power systems, one wishes to regulate the aggregate demand of an ensemble of distributed energy resources (DERs), such as controllable loads and battery energy storage systems. 
We suggest a notion of predictability and fairness, which suggests that the long-term averages of prices or incentives offered should be independent of the initial states of the operators of the DER, the aggregator, and the power grid.  
We show that this notion cannot be guaranteed with many traditional controllers used by the load aggregator, including the usual proportional-integral (PI) controller. 
We show that even considering the non-linearity of the alternating-current model, this notion of predictability and fairness can be guaranteed for incrementally input-to-state stable (iISS) controllers, under mild assumptions. 
\end{abstract}

\section{Introduction}
In response to large-scale integration of intermittent sources of energy \cite{lund2015review} in power systems, 
system operators seek novel approaches to balancing the load and generation.  
The aggregation of  distributed energy resources (DERs) and regulation of their aggregate power output has emerged as a viable balancing mechanism \citep[e.g.]{SG2010, Hiskens2011,pinson2014benefits} for many system operators.\footnote{
In Australia, for instance, it is estimated that there are  2.5 million rooftop solar installations with a total capacity of more than 10 GW and 73,000 home energy storage systems with a 1.1 GW capacity.}

For the transmission system operator, the load aggregator provides a single point of contact, which is responsible for the management of risks involved in the coordination of an ensemble of DERs and loads, and the legal and accounting aspects associated with modern grid codes and contracts governing the provision of ancillary services to the grid.

The load aggregator then coordinates an ensemble of participants, which can take many forms, including DERs \citep[e.g.]{SG2010a}, interruptible or deferrable loads \citep[e.g.]{Alagoz2013,mathieu2014arbitraging}, and eventually, vehicle-to-grid (V2G) services \citep[e.g.]{EV2018,lee2020adaptive}. 
Load aggregation is sometimes also known as a virtual power plant \cite{el2010virtual,crisostomi2014plug,8013070}.
% haberle2021control

Independent of the form of the ensemble, load aggregation introduces a novel feedback loop \cite{6197252,li2020real} into regulation problems in power systems.
Recent regulation\footnote{In the European Union, Regulation (EU) 2019/943 of the European Parliament and of the Council of 5 June 2019 on the internal market for electricity sets ``fundamental principles [to] facilitate aggregation of distributed demand and supply'', while requiring ``non-discriminatory market access'' and ``fair rules''. Its implementation in individual member states varies and is still in progress in many member states since mid-2021.} mandates the use of load aggregation, and hence makes the closed-loop analysis of load aggregation very relevant.
A recent pioneering study of the closed-loop aspects of load aggregation by Li et al. \cite{li2020real} leaves three issues open:
how to go beyond a linearisation of the physics of the alternating-current (AC) model?
How to model the uncertainty inherent in response to the control signal provided by the load aggregator?
Finally, what criteria of a good-enough behaviour of such a closed-loop system to consider?
The first issue is particularly challenging: While the consideration of losses in the AC model may be necessary \cite{baker2021solutions} for the sake of the accuracy of the model, it introduces a wide range of challenging behaviours \citep{754100,rogers2012power} in the non-linear dynamics. 

In this paper, we address all three issues in parallel. First, we suggest that good-enough behaviour involves ``predictability and fairness'' properties that are rarely considered in control theory and require new analytical tools. 
These include:
\begin{itemize}
\item[(a)] whether each DER and load, on average, receives similar treatment in terms of load reductions, disconnections, or power quality, 
\item[(b)] whether the prices or incentives depend on initial conditions, such as consumption at some point in the past,
\item[(c)] whether the prices, incentives, or limitations of service are stable quantities, preferably not sensitive to noise entering the system,
and predictable.
\end{itemize}
Building upon the recent work of \cite{ErgodicControlAutomatica}, we show that these criteria can be phrased in terms of properties of a specific stochastic model. In particular, these criteria are achieved whenever one can guarantee the existence of the {\em unique invariant measure} of a stochastic model of the closed-loop system. Thus, the design of feedback systems for deployment in the demand-response application should consider both the traditional notions of regulation and optimality and provide guarantees concerning the existence of the unique invariant measure.

Next, we show that many familiar control strategies, in straightforward situations, do not necessarily give rise to feedback systems that possess the unique stationary measure.  
In theory, we use the model of \cite{ErgodicControlAutomatica} to show that %considering $N$ loads with states $x_i, i=1,\ldots,N$ and 
 with any controller ${\mathcal C}$ that is  linear marginally stable with a pole $s_1 =e^{q j \pi}$ on the unit circle, where $q$ is a rational number, the closed-loop system cannot have a unique invariant measure.
In practice, we illustrate this behaviour in simulations using Matpower \cite{zimmerman2010matpower}.

As our main result, we show that one can design controllers that allow for predictability and fairness from the point of view of a participant, such as an operator of DERs, even in the presence of losses in the AC model. These results are based on a long work history on incremental input-to-state stability, which we have extended to a stochastic setting.
\begin{figure}[t!]
\includegraphics[width=\columnwidth]{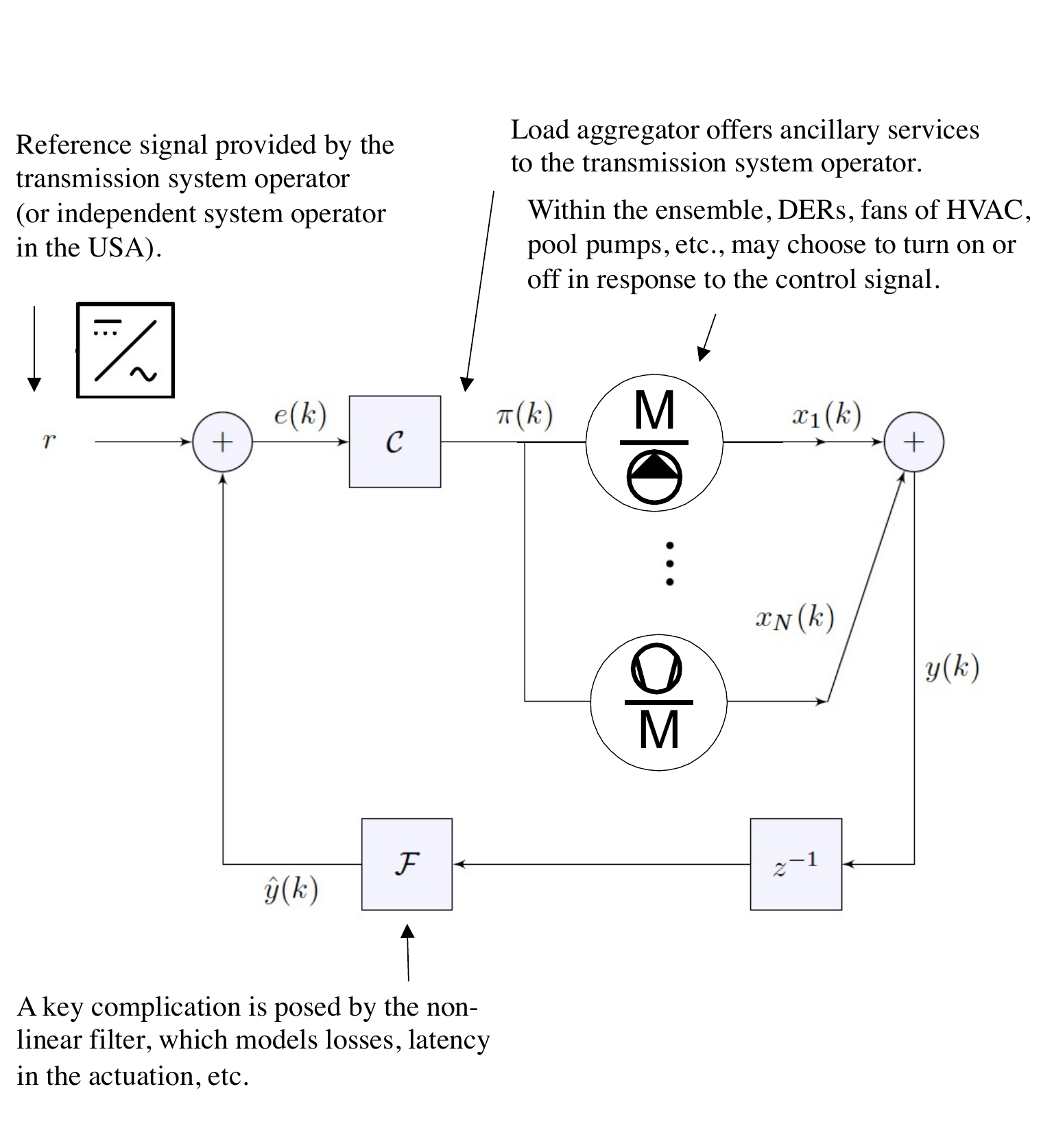}
\caption{An illustration of our closed-loop model, following \cite{ErgodicControlAutomatica}.}
\label{system}
\end{figure}

\section{Related Work}
There is a rich history of work on dynamics and control in power systems, summarised in several excellent textbooks of \cite{kundur1994power,sauer1998power,machowski2011power,rogers2012power}, and \cite{bevrani2014robust}. Without aiming to present a comprehensive overview, one should like to point out that traditionally, load frequency control (LFC) is performed using controllers with poles on the unit circle, such as proportional-integral (PI) controllers \citep[e.g.]{4074993,32469}.
%There is a substantial amount of work on the tuning \citep[e.g.]{sheirah1984improved, 5361327} of such controllers.
There are many alternative proposals, some of them excellent \cite[e.g.]{Ugrinovskii2005,Perfumo2012,zarate2010opf,li2016sqp,wang2018sdp,li2019sequential,nojavan2020voltage,nikkhah2020stochastic,9511198}, but even the most recent textbook of \cite[Chapter 4, ``Robust PI-Based Frequency Control'']{bevrani2014robust} suggests that ``complex state feedback or higher-order dynamic controllers [...] are impractical for [the] industry''.

Although the idea of demand response has been studied for several decades \citep[e.g.]{4112061}, only recently it has been shown to have commercial potential \citep{4275305,5608535,6558529, Biegel2014}. Notably, with a capacity as low as 20--70 kWh, load aggregators may break even \cite{Biegel2014} in the spot market in developed economies, although many grid codes still require load aggregators to aggregate a substantially more enormous amount of active power (e.g., 1 MW).  

Correspondingly, there has been much interest in incorporating  load aggregation in load frequency control \citep{5764847,6112198,6239581,6224203,6315670,6397579,6508915,6509996,6702462,6717056,6874591,6839081,dorfler2017gather,7004818,7065326,7944568}.
We refer to \cite{Dehghanpour2015} for a recent survey.

In a recent pioneering study, Li et al. \cite{li2020real} considered the closed-loop system design while quantifying the available flexibility from the aggregator to the system operator,
albeit without the non-linearity of the alternating-current model.

While some of the more recent approaches utilize model-predictive control \cite[e.g.]{7065326}, many of the early proposals extend the use of PI. In contrast, we utilise {\em iterated function systems} \citep{elton1987ergodic,BarnsleyDemkoEltonEtAl1988,barnsley1989recurrent}, a class of Markov processes, to design novel controllers following \cite{ErgodicControlAutomatica}.

\section{Notation, Preliminaries and Our Model}\label{sec:model}
Let us introduce some definitions and notation:
\begin{dfn}\label{df:class-K}
A function $\gamma : \mathbb R^{+}\to \mathbb R^{+}$ is is said to be of class $\mathcal K$ if it is continuous, increasing and $\gamma(0)=0$. It is of class $\mathcal K_{\infty}$ if, in addition, it is proper, i.e., unbounded.
\end{dfn}
\begin{dfn}\label{df:class-KL}
A continuous function $\beta: \mathbb R^{+}\times \mathbb R^{+}\to \mathbb R^{+}$ is said to be of class $\mathcal K \mathcal L$, if for all fixed $t$ the function $\beta(\cdot, t)$ is of class $\mathcal K$ and for all fixed $s$, the function $\beta(s, \cdot)$ is is non-increasing
and tends to zero as $t\to \infty$.
\end{dfn}
In keeping with \cite{ErgodicControlAutomatica}, our model is based on a feedback signal $\pi(k)\in \Pi \subseteq \mathbb{R}$ sent to $N$ agents.
Each agent $i$ has a state $x_i\in \R^{n_i}$ and an associated output $y_i \in\mathbb{D}_i$, where the latter is a finite set.
The non-de\-ter\-min\-is\-tic response of agent $i$ is within a finite set of actions:
$\mathbb{A}_i =\{
a_1,\ldots,a_{L_i} \}\subset \R^{n_i}$.
The finite set of possible resource demands of agent $i$ is
$\mathbb{D}_i$:
\begin{equation}
\label{eq:5}
\mathbb{D}_i := \{ d_{i,1},
d_{i,2}, \ldots, d_{i,m_i}\}.
\end{equation}
We assume there are $w_i \in \N$ state transition maps ${\mathcal W}_{ij}: \R^{n_i} \to \R^{n_i}$, $j=1,\ldots,w_i$
for agent $i$
and $h_i \in \N$ output maps ${\mathcal H}_{i\ell}: \R^{n_i} \to
\mathbb{D}_i$, $\ell= 1,\ldots,h_i$ for each agent $i$.
The evolution of the states and
the corresponding demands then satisfy:
\begin{align}\label{eq:general}
x_i(k+1) & \in  \{  {\mathcal W}_{ij}(x_i(k)) \;\vert\; j = 1, \ldots, w_i\}, \\
y_i(k) & \in  \{ {\mathcal H}_{i\ell}(x_i(k)) \;\vert\; \ell = 1, \ldots, h_i\},
\end{align}
where the choice of agent $i$'s response at time $k$ is governed by probability
functions $p_{ij} : \Pi \to [0,1]$, $j=1,\ldots,w_i$, respectively
$p'_{i\ell} : \Pi \to [0,1]$, $\ell=1,\ldots,h_i$. Specifically, for each agent $i$, we have
for all $k\in\N$ that
\begin{subequations} \label{eq:problaws}
\begin{align}\label{eq:problaw-1}
&\mathbb{P}\big( x_i(k+1) = {\mathcal W}_{ij}(x_i(k)) \big) = p_{ij}(\pi(k)),\\
&\mathbb{P}\big( y_i(k) = {\mathcal H}_{i\ell}(x_i(k)) \big) = p'_{i\ell}(\pi(k)).
\intertext{Additionally, for all $\pi \in \Pi$, $i=1,\ldots,N$ it holds that}
&\sum_{j=1}^{w_i} p_{ij}(\pi) = \sum_{\ell=1}^{h_i} p'_{i\ell}(\pi) = 1.
\end{align}
\end{subequations}
The final equality comes from the fact $p_{ij}$, $p'_{i\ell}$ are
probability functions. We assume that, conditioned on $\{ x_i(k)  \}, \pi(k)$,
the random variables $\{ x_i(k+1) \mid  i = 1,\ldots,N \}$ are stochastically independent. The outputs $y_i(k)$ each depend on $x_i(k)$ and the signal $\pi(k)$ only.

\label{sec:markov}
Let $\Sigma$  be a closed subset  of $\R^n$ with the usual Borel
$\sigma$-algebra $\mathbb{B}(\Sigma)$. We call the elements of $\mathcal{B}(\Sigma)$
events. A Markov chain on $\Sigma$ is a sequence
of $\Sigma$-valued random vectors $\{ X(k)\}_{k\in\N}$ with the Markov property, which is the equality of a probability of an event conditioned on past events
and probability of the same event conditioned on the current state, \emph{i.e.}, we always have
\begin{align}
\probability \left( X(k+1) \in \mathcal{G} \mid X(j)=x_{j}, \,  j=0, 1, \dots, k \right) \nonumber \\
= \probability\big(X(k+1)\in \mathcal{G} \mid X(k)=x_{k}\big), \nonumber
\end{align}
where $\mathcal{G}$ is an event
and $k \in \N$. We assume the Markov chain is time-homogeneous. The transition operator $\mathcal P$ of the Markov chain is defined
for $x \in \Sigma$, $\mathcal{G} \in \mathcal{B}\left(\Sigma\right)$ by
\begin{align}
P(x,\mathcal{G}) := \probability(X(k+1) \in \mathcal{G} \; \vert \; X(k) = x).    
\end{align}
If the initial condition $X(0)$ is distributed according to an initial distribution
$\lambda$, we denote by $\probability_\lambda$ the probability measure induced on
the path space, i.e., space of sequences with values in $\Sigma$. Conditioned on an
initial distribution $\lambda$, the random variable $X(k)$ is distributed
according to the probability measure $\lambda_k$ which is determined inductively by $\lambda_0=\lambda$ and
\begin{equation}
\label{eq:measureiteration}
\lambda_{k+1}(\mathcal{G})  := \int_{\Sigma} \mathcal P(x, \mathcal{G}) \, \lambda_k(d x),
\end{equation}
for $\mathcal{G} \in \mathbb{B}$. 

A probability measure $\mu$ is called
invariant with respect to the Markov process $\{ X(k) \}$ if it is a fixed
point for the iteration described by \eqref{eq:measureiteration},
i.e., if 
\begin{align}\label{df: inv-meas}
\mathcal P \mu = \mu\quad
\end{align}
An invariant probability measure $\mu$ is called attractive, if for every probability measure $\nu$ the sequence $\{\lambda_k \}$ defined by \eqref{eq:measureiteration} with initial condition $\nu$ converges to $\mu$ in distribution.

In an iterated function systems with place dependent probabilities \citep{Elton1987,Barnsley1988(1),Barnsley1989}, we are given a set of maps $\{ f_j : \Sigma \to \Sigma\; \vert \; j \in \mathcal{J} \}$, where $\mathcal{J}$ is a (finite or countably infinite) index set. Associated to these maps, there are probability functions
\begin{align}
p_j: \Sigma \to [0,1]\text{ with } \sum_{j\in \mathcal{J}} p_j(x) = 1 \text{ for all } x\in \Sigma.  
\label{df: prob-funct}
\end{align}
The state $X(k+1)$ at time $k+1$ is then given by $f_j(X(k))$ with probability $p_j(X(k))$,
where $X(k)$ is the state at time $k$. Sufficient conditions for the existence of a unique attractive
invariant measure can be given in terms of ``average contractivity'' \citep{elton1987ergodic,BarnsleyDemkoEltonEtAl1988,barnsley1989recurrent}.
Any discrete-time Markov chain can be generated by an \emph{iterated function system with probabilities} see Kifer\cite[Section 1.1]{Kifer2012} or \cite[Page 228]{Bhattacharya2009}, although such representation is not unique, see, Stenflo \cite{Stenflo1999}. Although not well known, iterated function systems (IFS) are a convenient and rich class of Markov processes. The class of stochastic systems arising from the dynamics of multi-agent interactions can be modelled and analysed using IFS in a particularly natural way. A wealth of results\citep{Elton1987,Barnsley1988(2),Barnsley1989,Barnsley2013,Stenflo2001(1),Szarek2003(1),Steinsaltz1999, Walkden2007,Barany2015,Diaconis1999,Iosifescu2009,Stenflo2012(s)} then applies.

\section{The Closed-Loop Model as Iterated Random Functions}
Let us now model the system of Figure \ref{system} by a Markov chain on a state space representing all the system components. In general, one could consider:
\begin{align}
\label{eq:nonlinear-agents-all}
\phantom{{\mathcal A}}  &\left\{ \begin{array}{ccl}
x_i(k+1) &\in& \{  {\mathcal W}_{ij}(x_i(k)) \; \vert \; j=1,\ldots, w_i \}  \\
y_i(k) &\in& \{  {\mathcal H}_{ij}(x_i(k))  \; \vert \; j=1,\ldots, h_i \}, \\
\end{array} \right. \\
& \hspace*{1.6cm} y(k) = \sum_{i=1}^N y_i(k),
\end{align}
\begin{equation}
\label{eq:nonlinear-filter}
{\mathcal F} ~:~ \left\{ \begin{array}{ccl}
x_f(k+1) &=& {\mathcal W}_{f}(x_f(k),y(k)) \\
\hat y(k) &=& {\mathcal H}_{f}(x_f(k),y(k)),
\end{array} \right.
\end{equation}
\begin{equation}
\label{eq:nonlinear-cont}
{\mathcal C} ~:~ \left\{ \begin{array}{ccl}
x_c(k+1) &=& {\mathcal W}_{c}(x_c(k),\hat y(k),r) \\
\pi(k) &=& {\mathcal H}_{c}(x_c(k),\hat y(k),r),
\end{array} \right.
\end{equation}
In addition, we have (e.g., Dini) continuous probability functions
$p_{ij},p'_{il}:\Pi \to [0,1]$ so that
the probabilistic laws \eqref{eq:problaws} are satisfied.
If we denote by $\mathbb{X}_i, i=1,\ldots,N, \mathbb{X}_C$ and $\mathbb{X}_F$ the state spaces of the
agents, the controller and the filter, then the system evolves on the
overall state space $\mathbb{X} := \prod_{i=1}^N \mathbb{X}_i \times \mathbb{X}_C \times \mathbb{X}_F$
according to the dynamics
\begin{equation}
\label{eq:6}
x(k+1) := \begin{pmatrix}
(x_i)_{i=1}^N  \\
x_f \\
x_c
\end{pmatrix} (k+1) \in \{ F_m(x(k)) \,\vert\, m \in {\mathbb M}\}.
\end{equation}
where each of the maps $F_m$ is of the form
\begin{equation}
\label{eq:F_m-definition}
F_m(x(k)) \defeq  \begin{pmatrix}
( {\mathcal W}_{ij}(x_i (k)) )_{i=1}^N  \\
{\mathcal W}_f(x_f(k), \sum_{i=1}^N  {\mathcal H}_{i\ell}(x_i (k))) \\
{\mathcal W}_c(x_c(k), {\mathcal H}_f(x_f(k), \sum_{i=1}^N  {\mathcal H}_{i\ell}(x_i (k))))
\end{pmatrix}
\end{equation}
and the maps $F_m$ are indexed by indices $m$ from the set
\begin{equation}
\label{eq:8}
\mathbb{M} \defeq \prod_{i=1}^N \{ (i,1), \ldots, (i,w_i) \} \times \prod_{i=1}^N \{ (i,1), \ldots, (i,h_i) \}.
\end{equation}
By the independence assumption on the choice of the transition maps and
output maps for the agents, for each multi-index
$m=((1,j_1),\ldots,(N,j_N),(1,l_1),\ldots,(N,l_N))$ in this set, the
probability of choosing the corresponding map $F_m$ is given by
\begin{multline}
\label{eq:9}
\probability \left( x(k+1)=F_m(x(k)) \right) = \\ \left(\prod_{i=1}^N
p_{ij_i}(\pi(k)) \right) \left( \prod_{i=1}^N p'_{il_i}(\pi(k))
\right) =: q_m(\pi(k)).
\end{multline}
We have obtained an iterated function system of the previous subsection \ref{sec:markov}.
\subsection{Our Objective}
We aim to achieve regulation, with probability $1$
\begin{equation}
\label{eq:1}
\sum_{i = 1}^N y_i(k) = y(k) \leq r
\end{equation}
to some reference $r>0$, given by the installed capacity.
Further, we aim for predictability, in the sense that for each agent $i$ there exists a
constant $\overline{r}_i$ such that
\begin{equation}
\label{eq:3}
\lim_{k\to \infty} \frac{1}{k+1} \sum_{j=0}^k y_i(j) = \overline{r}_i,
\end{equation}
where the limit is independent of initial conditions.
This can be guaranteed, when one has a unique invariant measure for the closed-loop model as iterated random functions.
The requirement of \textbf{fairness} is then for the limit $\overline{r}_i$ to coincide
for all $1 \le i \le N$.
\subsection{A Motivating Negative Result}\label{sec:comments-pi-negative}
The motivation for our work stems from the fact that controllers with integral action \citep{Franklin, Franklin_dig}, such as the Proportional-Integral (PI) controller, fail to provide unique ergodicity, and hence the fairness properties.
In many applications, controllers with integral action are widely adopted, including leading control systems \cite{barker2007speedtronic} in power generation.
A simple PI control can be implemented as:
\begin{equation}\label{pid1}
\pi(k) = \pi(k-1) + \kappa \big[ e(k) - \alpha e(k-1) \big],
\end{equation}
which means its transfer function from $e$ to $\pi$, in terms of the ${\mathcal Z}$ transform, is given by
\begin{equation}\label{pid2}
C(z) = \kappa\frac{1 - \alpha z^{-1}}{1 - z^{-1}}.
\end{equation}
Since this transfer function is not asymptotically stable, any associated realisation matrix will not be Schur.
Note that this is the case for any controller with any sort of integral
action, \emph{i.e.}, a pole at $z = 1$.
\begin{thm}[\cite{ErgodicControlAutomatica}]
\label{thm:pole}
Consider $N$ agents with states $x_i, i=1,\ldots,N$. Assume
that there is an upper bound $m$ on the different values the agents
can attain, \emph{i.e.}, for each $i$ we have $x_i \in \mathbb{A}_i =\{
a_1,\ldots,a_{m_i} \}\subset \R$
for a given set ${\mathbb{A}}_i$ and $1 \leq m_i \leq m$.

Consider the feedback system in Figure \ref{system}, where ${\mathcal
F}\, : \, y \mapsto \hat y$ is a finite-memory moving-average (FIR)
filter.  Assume the controller ${\mathcal C}$ is a linear marginally
stable single-input single-output (SISO) system with a pole $s_1 =
e^{q j \pi}$ on the unit circle where $q$ is a rational number,
$j$ is the imaginary unit,
and $\pi$ is Archimedes' constant $3.1416$. In
addition, let the
probability functions $p_{il} : \R \to [0,1]$ be continuous for all
$i=1, \ldots, N, l=1,\ldots,m_i$, \emph{i.e.}, if
$\pi(k)$ is the output of ${\mathcal C}$ at time $k$, then $\probability(x_i(k+1)=a_l)=p_{il}(\pi(k))$. Then
the following holds.
\begin{enumerate}
\item[(i)] The set ${\mathbb O}_{\mathcal F}$ of possible output values of the filter ${\mathcal F}$ is finite.
\item[(ii)] If the real additive group ${\mathcal E}$ generated by
$\{ r - \hat y \mid \hat y \in {\mathbb O}_{\mathcal F} \}$ with $r$ from \eqref{eq:1} is discrete, then
the closed-loop system cannot be uniquely ergodic.
\item[(iii)] For rational ${\mathbb A}_i
\subset \mathbb{Q}$ for all $i=1,\ldots,N$,  $r
\in \mathbb{Q}$, and rational coefficients of the FIR filter ${\mathcal F}$, the closed-loop
system cannot be uniquely ergodic.
\end{enumerate}
\end{thm}
This suggests that it is perfectly possible
for the closed loop both to perform its regulation function well and, at the same time, to
destroy the ergodic properties of the closed loop.
We demonstrate this in a realistic load-aggregation scenario in Section \ref{sec:numeric}.
\section{The Main Result}\label{sec:nonlinear}
For a linear controller and filter, \cite{ErgodicControlAutomatica} have established conditions that assure the existence of a unique invariant measure for the closed-loop. Let us consider non-linear controllers and filters, as required in power systems, and seek unique ergodicity conditions.
Incremental stability is well-established concept to describe the asymptotic property of differences between any two solutions. One can utilise the concept of incremental input-to-state stability, which is defined as follows:
\begin{dfn}[Incremental ISS, \cite{angeli2002lyapunov}]
Let $\mathcal U
%=:\{ u(k)\}_{k\in \mathbb Z_{\ge k_0}}
$ denote the set of all 
%bounded 
input functions $u: \mathbb Z_{\ge k_0}\to \mathbb R^d$%  
%and let state $x(k)\in \mathbb R^n \quad \forall k\in \mathbb Z_{\ge k_0}$
. Suppose, $F: \mathbb R^d \times \R^n\to \mathbb R^n$ is continuous, then the discrete-time non-linear dynamical system
\begin{align}
x(k+1) = F(x(k),u(k)),
\label{system-inciss}
\end{align}
is called (globally) \emph{incrementally input-to-state-stable} (incrementally ISS), if there exist $\beta\in \mathcal{KL}$ and $\gamma\in\mathcal K$ such that
for any pair of inputs $u_1, u_1\in\mathcal{U}$ and any pair of initial condition  $\xi_1, \xi_2 \in \R^n$:
\begin{align*}
&\|x(k, \xi_1, u_1)-x(k, \xi_2, u_2)\|\\ 
&\le \beta(\|\xi_1 - \xi_2\|, k)+ \gamma(\|u_1-u_2\|_{\infty}),\quad \forall k\in \N.   
\end{align*}
\end{dfn}
%\begin{dfn}(Solution Concept)
%For any $k_0\in\mathbb Z_{\ge k_0}$, $\xi \in \mathbb R^n$ and for any bounded, %admissible $u: \mathbb Z_{\ge k_0}\to \mathbb R^d$, a forward solution of the %system \eqref{system-inciss} is a function $\phi: \mathbb Z_{\ge k_0} \to %\mathbb R^n$, parameterized by initial state and time, satisfying %$\phi(k_0)=\xi$ and  $\phi(k+1)=F(k, \phi(k))$.
%\end{dfn}
\begin{dfn}[Positively Invariant Set]
A set $\mathcal X\subseteq \mathbb R^n$ is called positively invariant under the system  \eqref{system-inciss} if for any initial state $x(k_0)=\xi\in \mathcal X$  we have $\phi(k)\in \mathcal X \quad \forall k\ge k_0$.
\end{dfn}
\begin{dfn}[Uniform Exponential Incremental Stability] The system  \eqref{system-inciss} is uniformly exponentially incrementally
stable in a positively invariant set $\mathcal X$ if there exists $\kappa \ge 1$ and $\lambda >1$ such that for any pair of initial states $\xi_1, \xi_2 \in \mathcal X$, for any pair of  $u_1, u_2\in \mathcal U$, and for all $k\ge k_0$ the following holds
\begin{align}
&\|\phi (k,k_0,\xi_1, u_1)-\phi (k,k_0,\xi_2, u_2)\|\nonumber\\
&
%\le \kappa \|\xi_1-\xi_2\| %\lambda^{-\|u_1-u_2\|_{\infty}}.
\le \kappa \|\xi_1-\xi_2\| \lambda^{-(k-k_0)}\|u_1-u_2\|_{\infty}.
\end{align}
If $\mathcal X= \mathbb R^n$, then the system  \eqref{system-inciss} is called uniformly globally exponentially incrementally stable.
\end{dfn}
With this notation, we can employ contraction arguments for iterated function systems \cite{BarnsleyDemkoEltonEtAl1988} to prove:

\begin{thm} \label{thm02Ext}
Consider the feedback system depicted in Figure \ref{system}.
Assume that each agent
$i \in \{1,\cdots,N\}$ has a state governed by the non-linear iterated
function system
\begin{align}
\label{eq:nonlinear-agents}
x_i(k+1) &= {\mathcal W}_{ij}(x_i(k)) \\
y_i(k) &= {\mathcal H}_{ij}(x_i(k)),
\end{align}
where:
\begin{itemize}
\item[(i)] we have globally Lipschitz-continuous and continuously differentiable functions ${\mathcal W}_{ij}$ and ${\mathcal H}_{ij}$ with
global Lipschitz constant $l_{ij}$, resp. $l'_{ij}$
\item[(ii)]
we have Dini continuous probability functions
$p_{ij},p'_{il}$ so that
the probabilistic laws \eqref{eq:problaws} are satisfied
\item[(iii)] there are  scalars $\delta, \delta' > 0$ such that
$p_{ij}(\pi) \geq  \delta > 0$,
$p'_{ij}(\pi) \geq  \delta' > 0$ for all $\pi\in \Pi$ and all $(i,j)$
\item[(iv)] the composition $F_m$ 
\eqref{eq:F_m-definition} of agents' transition maps and the probability functions satisfy the contraction-on-average condition. Specifically, consider the space $\mathbb{X}_a := \prod_{i=1}^{N} \mathbb {X}_i$ of all agents and the maps 
$F_{a,m}: \mathbb{X}_a \to \mathbb{X}_a$, $(x_i)_{i=1}^N\mapsto (\mathcal{W}_{im}(x_i))_{i=1}^N$ together with the probabilities $p_m:\Pi\to[0,1]$, $m\in \mathbb{M}$, form an average contraction system in the sense that there exists a constant $0<\tilde c<1$ such that for all $x,\hat{x} \in \mathbb{X}_a, x\neq\hat{x}, \pi\in\Pi$ we have
\begin{equation}
\label{eq:average-contraction}
\sum_{m}p_m(\pi) \frac{\|F_{a,m}(x)-F_{a,m}(\hat{x})\|}{\|x-\hat{x}\|} < \tilde c.
\end{equation}
\end{itemize}
Then, for every incrementally input-to-state stable controller $\mathcal{C}$ and every incrementally input-to-state stable filter $\mathcal{F}$ compatible with the feedback structure, the feedback loop has a unique, attractive
invariant measure. In particular, the system is uniquely ergodic.
\end{thm}
\begin{proof}
To prove the result, we have to show that the assumptions of Theorem~2.1 in \cite{BarnsleyDemkoEltonEtAl1988} are satisfied. To this end, we note that the maps $F_m$ satisfy, by (i), the necessary Lipschitz conditions. Further, the probability maps have the required regularity and positivity properties by (ii) and (iii). Thus, it only remains to show that the maps $F_m$ defined in \eqref{eq:F_m-definition} have the average contraction property. The existence of a unique attractive invariant measure then follows from \cite{BarnsleyDemkoEltonEtAl1988} and unique ergodicity follows from \cite{elton1987ergodic}.

First we note that the filter and the controller are incrementally stable systems that are in cascade. It is 
%straightforward to see 
shown in \cite[Theorem 4.7]{angeli2002lyapunov} that cascades of incrementally ISS systems are incrementally ISS. The proof is given for continuous-time systems but readily extends to the discrete-time case. It will therefore be sufficient to treat filter and controller together as one system with state $z$.

To show the desired contractivity property, consider two distinct points $x=(x_a,z)=((x_i),z), \hat{x}=(\hat{x}_a,\hat{z})=((\hat{x}_i),\hat{z})\in \mathbb{X}$. Let $L$ be an upper bound for the Lipschitz constants of the output functions $\mathcal{H}_{ij}$, $\mathcal{H}_f$. The assumptions on incremental ISS ensure that for each index $m$
\begin{multline*}
\|F_m(x) - F_m(\hat{x}) \| \leq \\ \max \{ \|F_{a,m}(x_a)  - F_{a,m}(\hat{x}_a) \|, \\ \beta_f(\|z-\hat{z}\|,1) + \gamma_f(L\|x_a - \hat{x}_a\|) \}.
\end{multline*}
As the maps $F_{a,m}$ satisfy the average contractivity condition, it is an easy exercise to see that also the sets of iterates satisfy the contractivity condition. The result then can be obtained using Theorem 15 of \cite{tran2018convergence}, where it is shown that incrementally stable systems are contractions,
using the converse Lyapunov theorem of \cite{jiang2002converse}.
The proof is concluded by invoking Theorem~2.1 in \cite{BarnsleyDemkoEltonEtAl1988}.
\end{proof}

Subsequently, one may wish to design a controller that is \emph{a priori} incrementally input-to-state-stable, e.g. considering extending the work on ISS-stable model-predictive controller \cite{1185106} or model-predictive controllers \cite{zarate2010opf,li2016sqp,wang2018sdp,li2019sequential,nojavan2020voltage,nikkhah2020stochastic,9511198} that are constrained to be small-signal stable. Alternatively, one may ask whether other controllers guaranteeing fairness are available, without explicitly adding incrementally input-to-state-stable constraints.
The implementation depends on the setting under consideration. 
%In this and many other models, e.g., obtained as a subset thereof, there are at least three approaches to implementing the constraints. 

\section{Numerical results}\label{sec:numeric}
Our numerical results are based on simulations utilising Matpower \cite{zimmerman2010matpower}, a standard open-source toolkit for power-systems analysis in version 7.1 running on Mathworks Matlab 2019b.
We use Matpower's power-flow routine to implement a non-linear filter, which models the losses in the alternating-current model. 

The model consists of an ensemble of $N$ DERs or partially-controllable loads divided into two groups that differ in 
their probabilistic response to the control signal provided by the load aggregator.
In particular, following the closed-loop of Figure \ref{system}:
The response of each system $\mathcal S_i$ to the control signal $\pi(k)$ provided by the load aggregator at time $k$ takes the form of 
probability functions, which suggest the probability a DERs is committed at time $k$ as a function of the control signal $\pi(k)$.
In our simulations, the two functions satisfying \eqref{df: prob-funct} are:
\begin{align}\label{eq:prob-func}
g_{i1}(x_i (k+1)=1)&=0.02+\frac{0.95}{1+\exp(-\xi\pi(k)-x_{01})} \\
g_{i2}(x_i (k+1)=1)&=0.98-\frac{0.95}{1+\exp(-\xi\pi(k)-x_{02})}
\end{align}
where we have implemented both PI and its lag approximant for the control of an ensemble of DERs or partially-controllable loads.
The binary-valued vector:
\begin{align}\label{eq:cmmtd-gnrtr}
u(k)\in \{0,1\}^N
\end{align} 
captures the output of the probability functions \eqref{eq:prob-func} in a particular realization. 

The commitment of the DERs within the ensemble is provided to Matpower, which using a standard power-flow algorithm computes 
the active power output $P(k)$ of the individual DERs:
\begin{align}\label{eq:output-gnrtr}
P(k)\in \mathbb{R}^N 
\end{align} 
which is aggregated into the total active power output $p(k) = \sum P(k)$ of the ensemble.
Then, a filter $\mathcal F$ is applied: 
\begin{align}\label{eq:filtr}
\hat p(k)= \frac{p(k)+p(k-1)}{2} + \textrm{losses}(k).
\end{align}
Notice that here we both smooth the total active power output with a moving-average filter and accommodate the losses in the AC model.
The error
\begin{align}\label{eq:err}
e(k)=r-\hat p(k)    
\end{align}
between the reference power output $r$ and the filtered value $\hat p(k)$ is then used as the input for the controller. 
Output of the controller $\pi(k)$ is a function of error $e(k)$ and an inner state of the controller $x_c (k)$. 
Signal of the controller is given by the PI or its lag approximant:
\begin{align}\label{eq:two-contrlr}
\pi_{\text{PI}}(k+1)&=\left[K_p e(k)+K_i \left(x_c (k)+e(k)\right)\right]\\
\pi_{\text{Lag}}(k+1)&=\left[K_p e(k)+K_i \left(0.99x_c (k)+e(k)\right)\right]
\end{align}
The procedure is demonstrated in the following pseudocode.
\begin{algorithm}[bht]
\SetAlgoLined
$P(0) = \textrm{power-flow}(u(0))$ for a given initial commitment $u(0)$ \\
\For{$i\gets0$ \KwTo $\infty$, i.e., each run of the simulation }{ 
\For{$i\gets1$ \KwTo $k_{\max}$, where the length $k_{\max}$ of the time horizon is given}{ 
		$p(k) = \sum P(k-1)$, cf. \eqref{eq:output-gnrtr} \\
		$\hat p(k) = {\mathcal F}(p(k))$, e.g. \eqref{eq:filtr}  \\
		$e(k) = r - \hat p(k)$, cf. \eqref{eq:err}  \\
		$\pi(k) = {\mathcal C}(e(k))$, cf. \eqref{eq:two-contrlr} \\
		$u(k) = {\mathcal S}(\pi(k))$, cf. \eqref{eq:prob-func} and \eqref{eq:cmmtd-gnrtr} \\
		$P(k) = \textrm{power-flow}(u(k))$, cf. \eqref{eq:output-gnrtr}
	}
 }
 \caption{Pseudo-code capturing the closed loop.}
\end{algorithm}

\subsection{A Serial Test Case}

\begin{figure}
\centering 
\includegraphics[width=0.75\columnwidth]{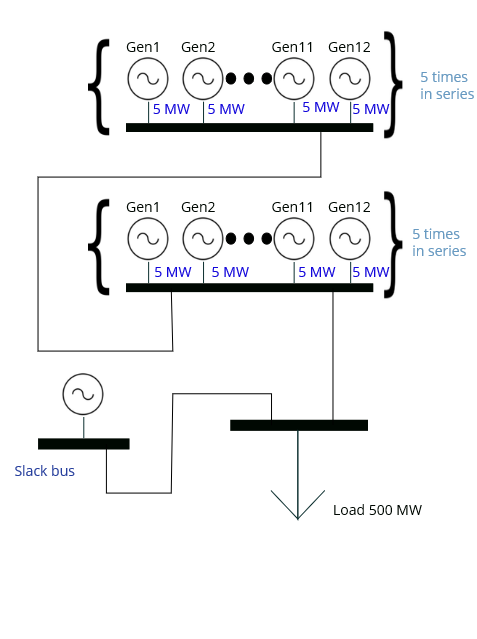}
\caption{A single-line diagram of the serial test case.}
\label{fig:distribution}
\end{figure}

First, let us consider a simple serial test case depicted in Figure \ref{fig:distribution}, where there are a number of buses connected in series, as they often are in distribution systems. In particular, the buses are: 
\begin{itemize}
    \item a slack bus, such as a distribution substation,  
    \item 5 buses with 12 DERs of 5 MW capacity connected to each bus, 
with probability function $g_{i2}$ of \eqref{eq:prob-func} modelling the response of these first 60 DERs to the signal,
\item 5 buses with 12 DERs of 5 MW capacity connected to each bus, with probability function $g_{i1}$ of \eqref{eq:prob-func}
modelling the response of these  60 DERs to the signal,
\item a single load bus with a demand of 500 MW.
\end{itemize}

We aim to regulate the system to the reference signal $r$, which is $300$ MW plus losses at time $k=0$.
Initially, the first 60 ($= 5 \cdot 12$) generators are off and the second following 60 ($= 5 \cdot 12$)
generators are on. 
%and similarly set $x_0$ to be $300$ plus losses at time $k=0$.
Considering the ensemble is composed of 120 DERs, with a total capacity of 600 MW,
this should yield a load of less than 200 MW on the slack bus.

We repeated the simulations 300 times, considering the time horizon of 2000 time steps each. The results in 
Figures \ref{fig:linear03main1}--\ref{fig:linear03main3} present the mean over the 300 runs with a solid line and the region of mean $\pm$ standard deviation as a shaded area.
Notice that we are able to regulate the aggregate power output (cf. Figure ~\ref{fig:linear03main1}). With the PI controller, however, the state of the controller, and consequently the signal and the power generated at individual buses over the time horizon, are determined by the initial state of the controller (cf. Figures \ref{fig:linear03main2}--\ref{fig:linear03main3}).

\begin{figure*}[h]
\centering 
\includegraphics[width=0.45\textwidth]{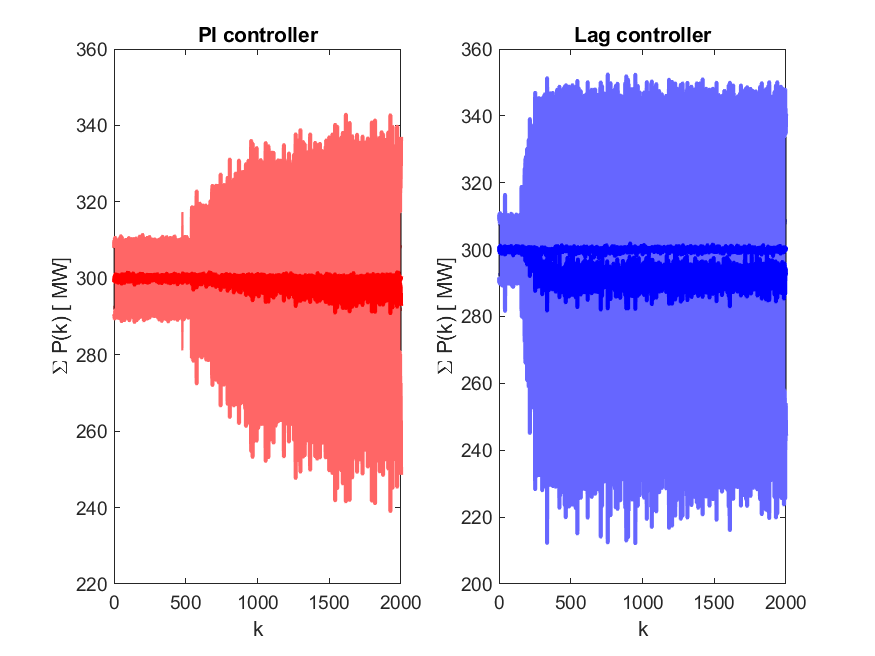}
\includegraphics[width=0.45\textwidth]{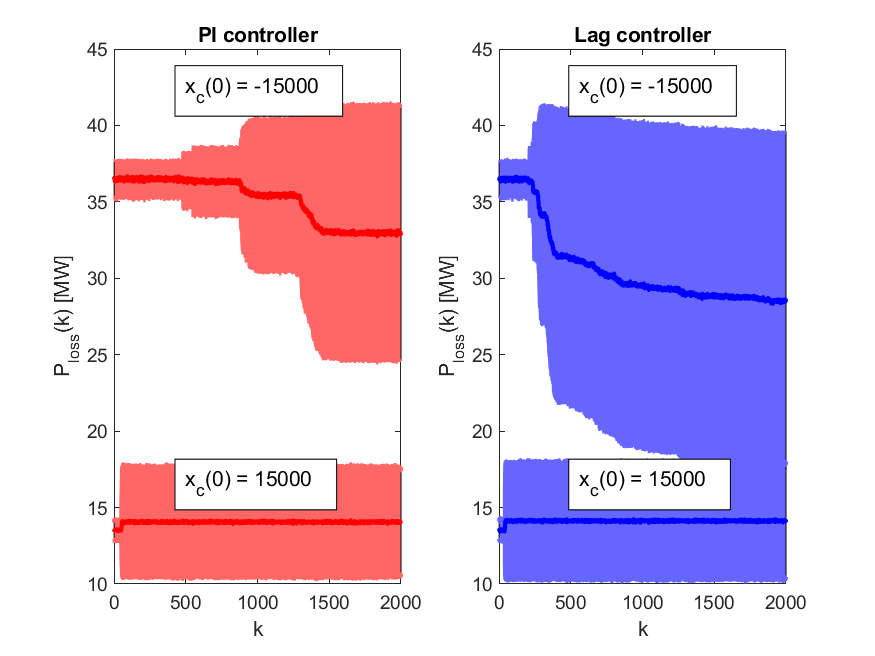}
\caption{Results of simulations on the serial test case regulated to 300 MW plus losses: 
Aggregate power produced by the ensemble (left) and the corresponding losses in transmission (right) as functions of time for the two controllers and two initial states of each of the two controllers.}
\label{fig:linear03main1}
\end{figure*}

\begin{figure*}[h]
\centering 
\includegraphics[width=0.45\textwidth]{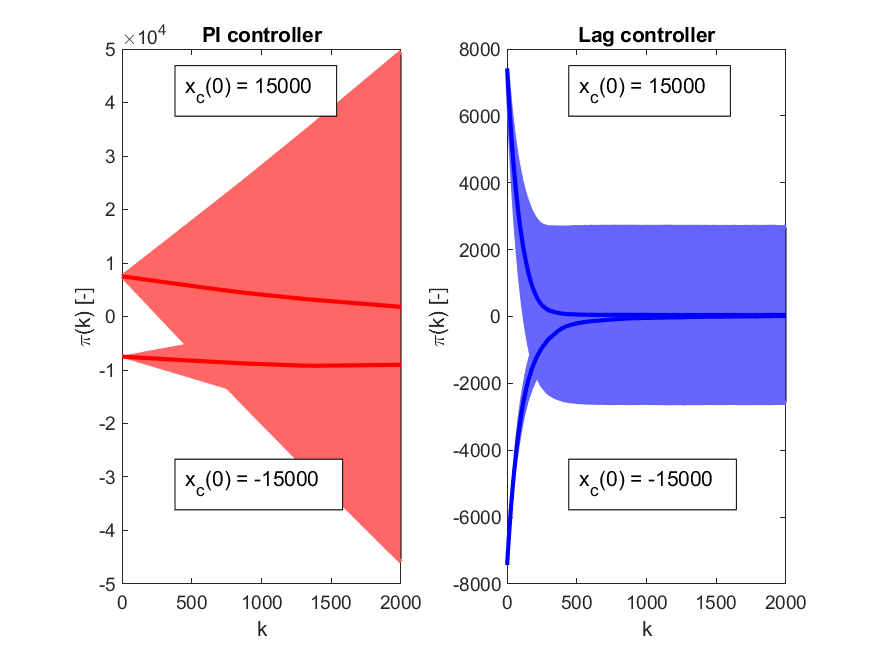}
\includegraphics[width=0.45\textwidth]{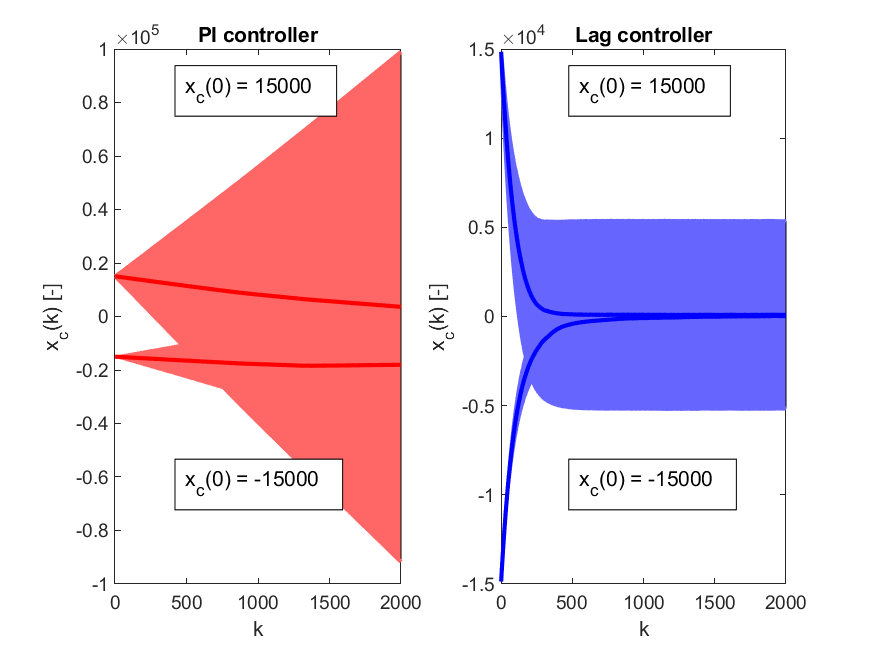}
\caption{Results of simulations on the serial test case:
Control signal (left) and the state of the controllers (right) as functions of time, for the two controllers and two initial states of each of the two controllers.}
\label{fig:linear03main2}
\end{figure*}

\begin{figure*}[h]
\centering 
\includegraphics[width=0.45\textwidth]{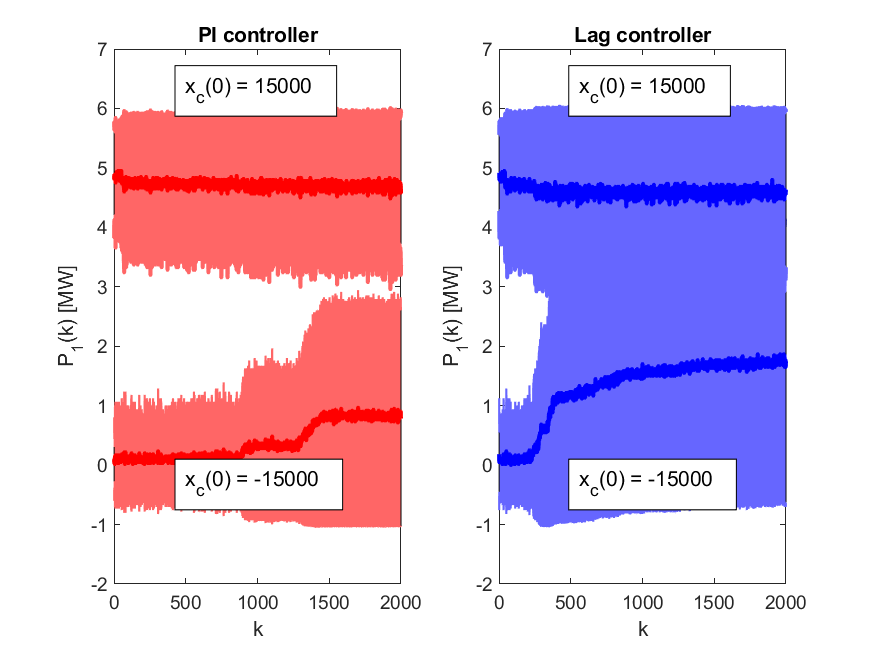}
\includegraphics[width=0.45\textwidth]{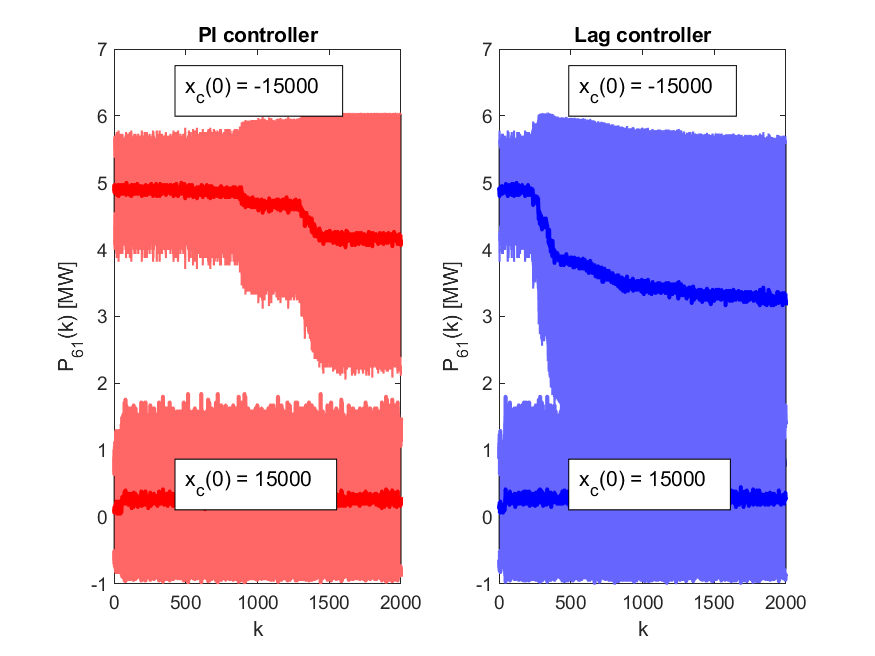}
\caption{Results of simulations on the serial test case: 
Powers at the first 60 DERs (left) utilising probability functions $g_{i2}$ of \eqref{eq:prob-func} and following 60 DERs (right) utilising probability functions $g_{i1}$ of \eqref{eq:prob-func}.}
\label{fig:linear03main3}
\end{figure*}

\begin{figure}
\centering 
\includegraphics[width=\columnwidth]{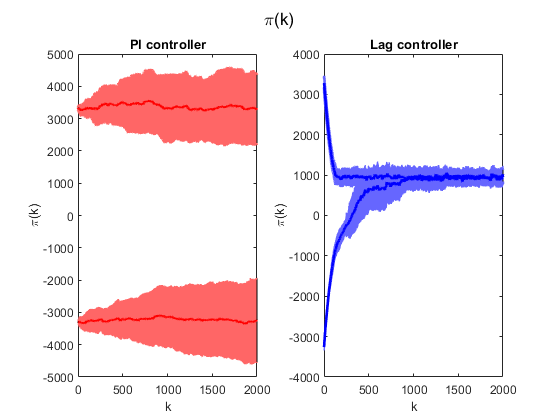}
\caption{Results of simulations on IEEE 118-bus test case: 
Control signal, as a function of time, for the two controllers and two initial states of each of the two controllers.}
\label{fig:sim2sig}
\end{figure}

\subsection{A Transmission-System Test Case}

Next, let us demonstrate the analogous results on the standard IEEE 118-bus test case with minor modifications. Notably, the ensemble is connected to two buses of the transmission system. At bus 10, 1 generator with a maximum active power output of 450 MW was replaced by 8 DERs with 110 MW of power each, out of which 4 DERs used probability function $g_{i1}$, where $x_{01}=660$.
The other four DERs used function $g_{i2}$, where $x_{02}=110$ and $\xi=100$. 
At bus 25, 1 generator with maximum active power output of 220 MW was replaced by four generators with active power output of 110 MW, out of which
two generators used probability function $g_{i1}$, and the other two used function $g_{i2}$ with $r=660$.
%As in the original 118-bus test case, bus 69 is the slack bus. % similarly to the generator at bus 3 in the  example. 

As in the serial test case, we have repeated the simulations 500 times, considering the time horizon of 2000 time steps each. 
Figure \ref{fig:sim2sig} yields analogical results to 
Figure \ref{fig:linear03main2} in the serial test case, where a different initial state $x_c (0)$ of the controller yields a very different control signal produced by the PI controller. This is not the case for the lag approximant. 
Consequently, some DERs may be activated much more often than others, all else being equal, solely as an artifact of the use of PI control.
Many further simulation results illustrating this behaviour are available in the Supplementary material on-line.

\section{Conclusions and Further Work}
We have introduced a notion of predictability of fairness in load aggregation and demand-response management, which relies on the concept of unique ergodicity \cite{9445023}, which in turn is based on the existence of a unique invariant measure for a stochastic model of a closed-loop model of the system.
Related notions of unique ergodicity have recently been used in social-sensing \cite{9445023}, and two-sided markets \cite{ghosh2021unique}; we envision it may have many further applications in power systems. 

The notions of predictability and fairness, which we have introduced, may be violated, if controls with poles on the unit circle (e.g., PI) are applied to load aggregation. 
In particular, the application of the PI controller destroys the ergodicity of the closed-loop model. 
Considering the prevalence of PI controllers within load frequency control, this is a serious concern.
On the other hand, we have demonstrated that certain other controllers guarantee predictability and fairness. 

An important direction for further study is to consider unique ergodicity of stability-constrained ACOPF \cite{zarate2010opf,li2016sqp,wang2018sdp,li2019sequential,nojavan2020voltage,nikkhah2020stochastic,9511198}, or related controllers based on stability-constrained model-predictive control (MPC). We believe there are a wealth of important results ahead.  

\paragraph*{Acknowledgements}
Ramen and Bob were supported in part by Science Foundation Ireland grant 16/IA/4610.
Jakub has been supported by OP VVV project CZ.02.1.01/0.0/0.0/16\_019/0000765 ``Research Center for Informatics''.
\FloatBarrier

\bibliography{ref,mdps,power}

\clearpage
\onecolumn
\normalsize
\appendix

For the convenience of the reader, we attach an overview of the notation and additional numerical results in the following appendices. 

\subsection{An Overview of Notation}

\begin{tabularx}{\linewidth}{ l | X }
\caption{A Table of Notation}\\\toprule\endfirsthead
\toprule\endhead
\midrule\multicolumn{2}{l}{\itshape continues on next page}\\\midrule\endfoot
\bottomrule\endlastfoot
\textbf{Symbol}                        & \textbf{Meaning}\\\midrule
$\mathbb{N}$                           & the set of natural numbers.\\
$\mathbb{Q}$                           & the set of rational numbers.\\
$\mathbb{R}$                           & the set of real numbers.\\
${\mathbb X}_i$                        & states-space of agent $i$.\\
${\mathbb X}_{\mathcal C}$             & states-space of controller.\\
${\mathbb X}_{\mathcal F}$             & states-space of filter.\\
${\mathbb X}$                          & overall states-space of the closed loop system.\\
${\mathbb Z}$                          & set of all integers.\\
$\mathcal G$                           & an event i.e an element in $\mathcal{B}\left(\Sigma\right)$.\\
$\mathcal J$                           & an index set.\\
$\Sigma$                               & a closed subset of $\mathbb R^n$.\\
$\mathcal{B}\left(\Sigma\right)$       & A Borel $\sigma$ algebra on $\Sigma$, consisting of all possible events.\\
$\{X_k\}_{k\in\N}$                     & A discrete-time Markov chain with state-space $\mathcal K$.\\
$\mathbf{P}$                           & A transition-probability operator.\\
$\mathbb P_{\lambda}$                  & the probability measure induced on
the path space.\\
$\lambda_k$                            & probability measures.\\
$\mu$                                  & a probability measure.\\
$j$                                    & an element in $\mathcal J$.\\
$\ell$                                 & cardinality of output maps.\\
$f_j$                                  & a function on $\mathcal K$.\\
$p_j$                                  & probability functions.\\
$\pi(k)$                               & control signal.\\
$\Pi$                                  & space of all control signal.\\
$\mathbb A_i$                          & space of all agent's actions.\\
$a_1,a_2,\dots, a_{L_i}$               & possibles actions.\\
$L_i$                                  & cardinality of action space of agent $i$.\\
$\mathbb D_i$                          & demand space of agent $i$.\\
$d_{i,1},\dots, d_{i,m_i}$             & possible demands.\\
$m_i$                                  & cardinality of demand space of agent $i$.\\
$n_i$                                  & dimension of the state-space of $i^{th}$ agent's private state-space.\\
$w_i$                                  & an affine map.\\
$h_i$                                  & number of output maps.\\
$\mathcal H_{i\ell}$                   & output maps.\\
$y_i(k)$                               & output of agent $i$.\\
$\hat y_(k)$                           & value of $y_(k)$ filtered by filter.\\
$p_{i\ell}(\pi)$                       & probability functions.\\
$p'_{i\ell}(\pi)$                      & probability functions.\\
$x_i(k)$                               & internal state of agent $i$ at time $k$.\\
$\mathcal W_{ij}$                      & transition maps.\\
$\mathcal F$                           & a filter.\\
$\mathcal C$                           & a controller.\\
$x_f(k)$                               & the internal state of the filter at time $k$ of dimension $n_f$.\\
$x_c(k)$                               & the internal state of the filter at time $k$ of dimension $n_c$\\
$\mathcal W_{c}$                       & map moddeling the controller.\\
$\mathcal H_{c}$                       & map moddeling the controller\\
$\epsilon$                             & a sufiiciently small positive real number\\
$l_m$                                  & Lipschitz constants.\\
$\nu^{\star}$                          & an invariant distribution.\\
$\beta$                                & a class $\mathcal {K L}$ function.\\
$\gamma$                               & a class $\mathcal K$ function.\\
$u(k)$                                 & a binary vector indicating the commitment of DERs in \eqref{eq:cmmtd-gnrtr}.\\
$p(k)$                                 & filtered value of the active power output of DERs in \eqref{eq:filtr}.\\
$P(k)$                                 & active power output of DERs in \eqref{eq:output-gnrtr}.\\
$\hat p(k)$                           &  aggregate active power output of the ensemble at time $k$.\\
$e(k)$                                 & error defined in \eqref{eq:err}.
\end{tabularx}

\clearpage

\subsection{Complete Numerical Results}

We provide numerical results for three test cases, two of which have been discussed in the main body of the text.
First, however, we present results for one additional, intentionally simple test case. 

\subsubsection{A Simple Test Case with Transmission Losses in the Filter}

The first one is a small 3-bus system depicted in Figure 1. 
At bus 1, there are $N = 10$ DERs, 
five of which use probability function  $g_{i1}$ with $x_{01}=200$. 
The other five DERs at bus 1 use function $g_{i2}$, where $x_{01}=40$ and $\xi =100$. 
Bus 2 is considered slack; therefore, the generator at bus 2 balances the power deviation of bus 1 and losses in order to fulfil power demand at bus 3. 
The simulations were run for both PI and its lag approximant for several different initial values of the controller's internal state $x_c (0)$. Figures \ref{fig:sim1power2}--\ref{fig:sim1sig} demonstrate that both controllers can regulate the ensemble around required power $r=200$. However, the PI controller's behaviour is strongly dependent on the initial condition of its internal state, whereas the lag approximant provides predictable behaviour in the long run. These results are in good agreement with simulation results in \cite{ErgodicControlAutomatica}.

\begin{figure}[h]
\centering 
\includegraphics[width=0.4\textwidth]{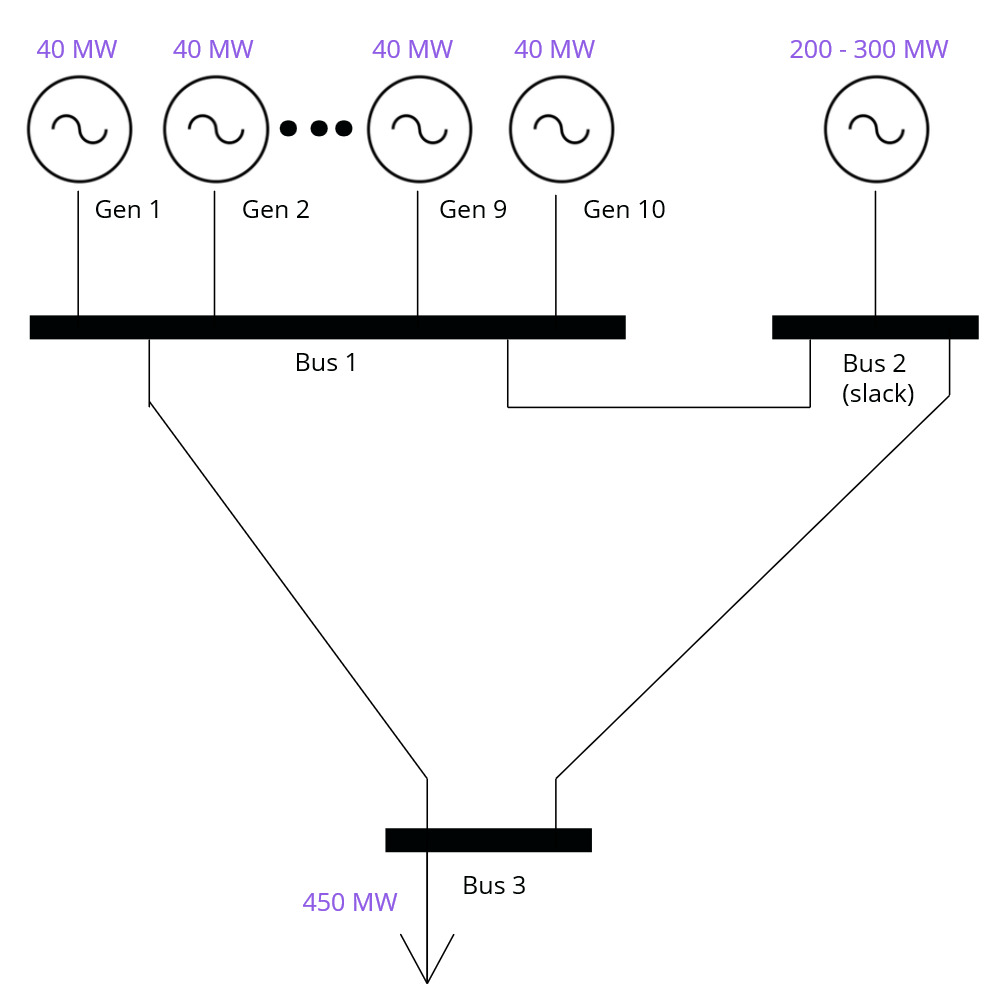}
\caption{A simple test case.}
\label{fig:sim1system}
\end{figure}

\begin{figure}
\centering 
\includegraphics[width=0.45\textwidth]{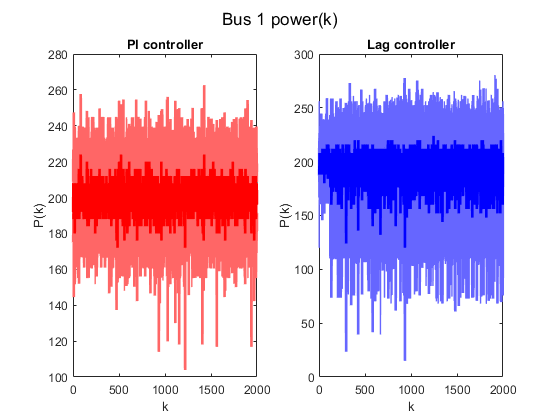} 
\includegraphics[width=0.45\textwidth]{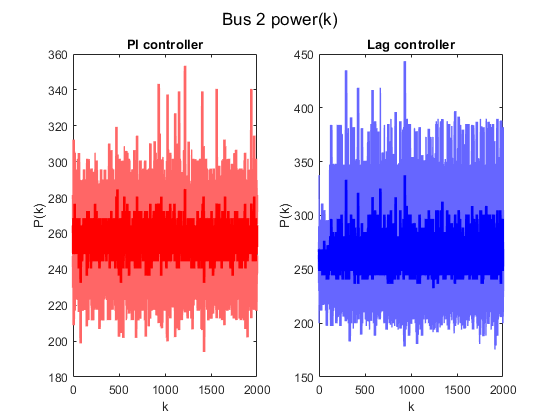}
\caption{Results of simulations on the system in Figure~\ref{fig:sim1system}: Power on buses 1 and 2 as a function of time
for the two controllers and two initial states $x_c$ of each of the two controllers.
Notice that both PI and its lag approximant are able to regulate the aggregate power produced by the ensemble of loads.}
\label{fig:sim1power2}
\end{figure}

\begin{figure}
\centering 
\includegraphics[width=0.45\textwidth]{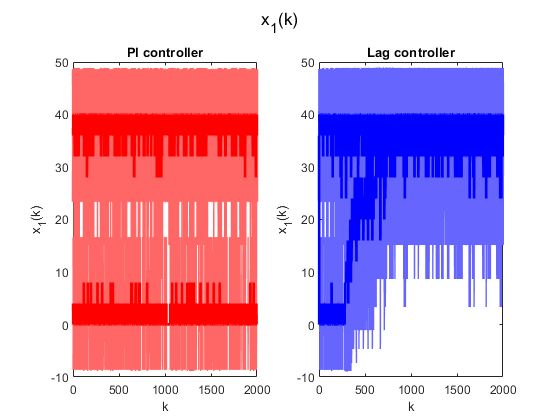}
\includegraphics[width=0.45\textwidth]{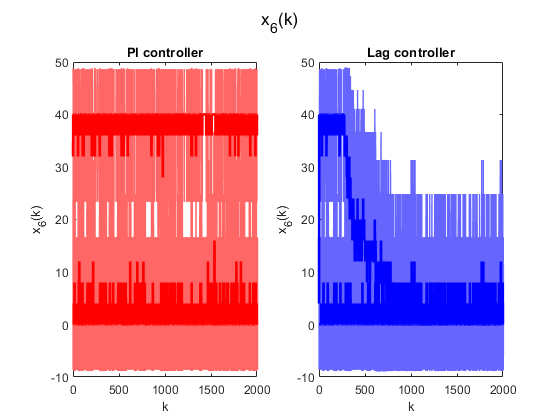}
\caption{Results of simulations on the system in Figure~\ref{fig:sim1system}: Observations available to the controller, as a function of time,
for the two controllers and two initial states $x_c$ of each of the two controllers.
Notice that both PI and its lag approximant receive similar inputs.}
\label{fig:sim1obs}
\end{figure}

\begin{figure}
\centering 
\includegraphics[width=0.45\textwidth]{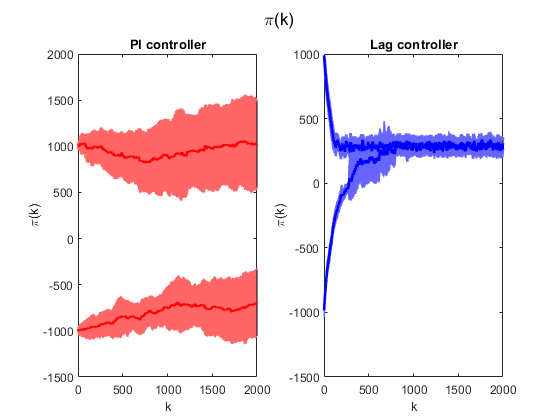}
\includegraphics[width=0.45\textwidth]{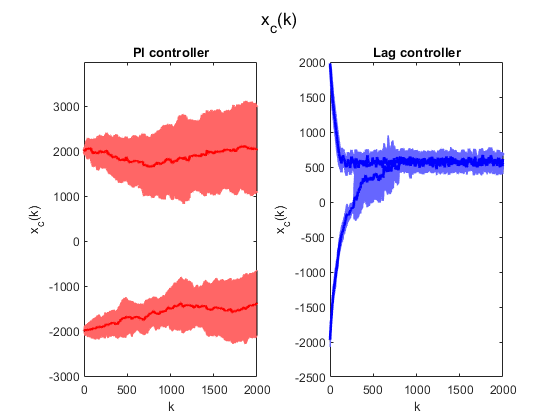}
\caption{Results of simulations on the system in Figure~\ref{fig:sim1system}: 
Control signal and the state of the controllers, as functions of time, for the two controllers and two initial states of each of the two controllers.}
\label{fig:sim1sigcon}
\end{figure}

\FloatBarrier
\subsection{A Serial Test Case with Losses in the Filter}

In our second example, we consider a serial test-case, with a number of buses connected in series. (A line-graph topology, in the language of graph theory.)
The buses are, in this order, 
\begin{itemize}
\item a slack bus, which could represent a distribution substation  
\item 5 buses with 12 DERs of 5 MW capacity each, 
and the ``type 2'' probability function. 
\item 5 buses with 12 DERs of 5 MW capacity each, 
and the ``type 1'' probability function. 
\item a single load bus with a 500 MW demand.
\end{itemize}

\begin{figure}[h]
\centering 
\includegraphics[width=0.4\textwidth]{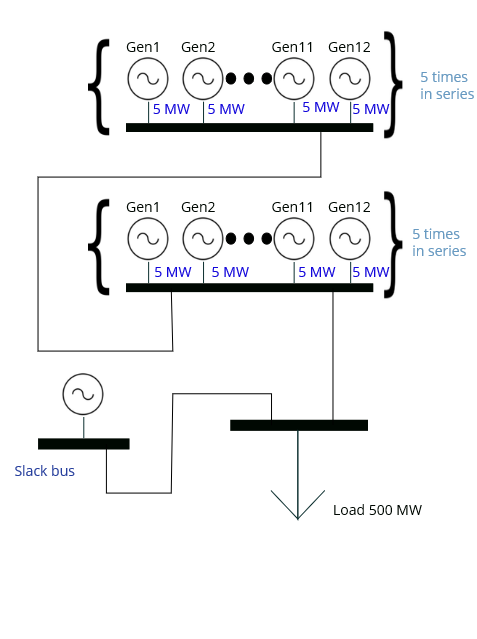}
\caption{A serial test case.}
\label{fig:sim2system}
\end{figure}

Overall, the ensemble is composed of 120 DERs, with a total capacity of 600 MW.
Initially, the first $5 \cdot 12$ generators are off and the second following $5 \cdot 12$
generators are on. 
We aim to regulate the power output of the ensemble to 60 running DERs (totalling 300 MW),
which should yield a load exceeding 200 MW on the slack bus (considering the losses).

In our first set of simulations on the serial test case, we repeat 500 times the simulations considering the time horizon of 2000 time steps each. $r=300$, $x_0 = 300$. The filter is $(y(k-1) + y(k))/2 + \textrm{losses}(k)$, whereby the losses hence propagate into the signal $\pi(k)$. 

As in the simple test case, we are able to regulate the aggregate power output (cf. Figure ~\ref{fig:linear02p1}). With PI controller, however, the state of the controller, and consequently the signal and the power at individual buses are determined by the initial state of the controller (cf. Figures \ref{fig:linear02p2}--\ref{fig:linear02p3}).

\begin{figure}[h]
\centering 
\includegraphics[width=0.45\textwidth,trim=0 0 0 10mm,clip]{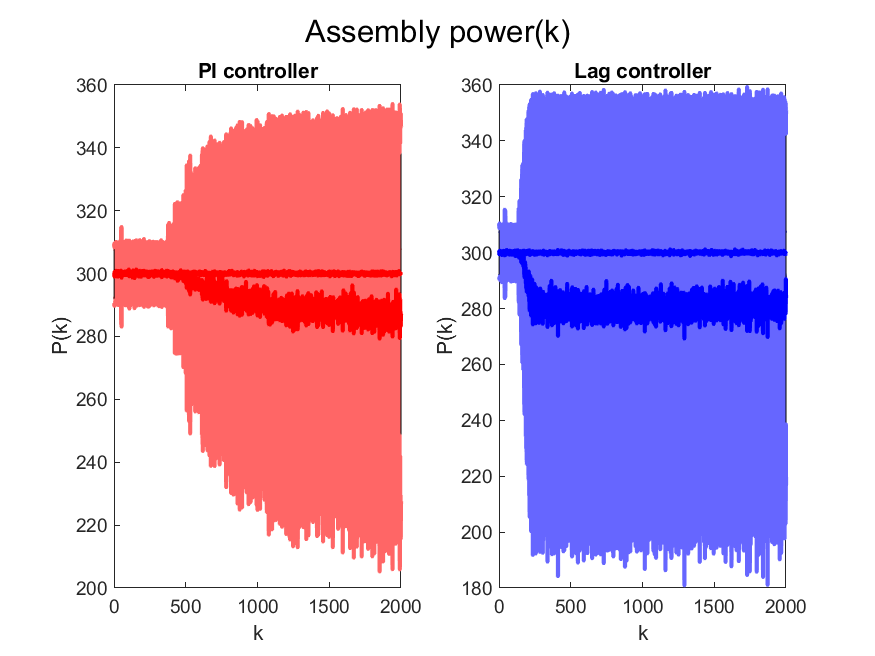}
\includegraphics[width=0.45\textwidth]{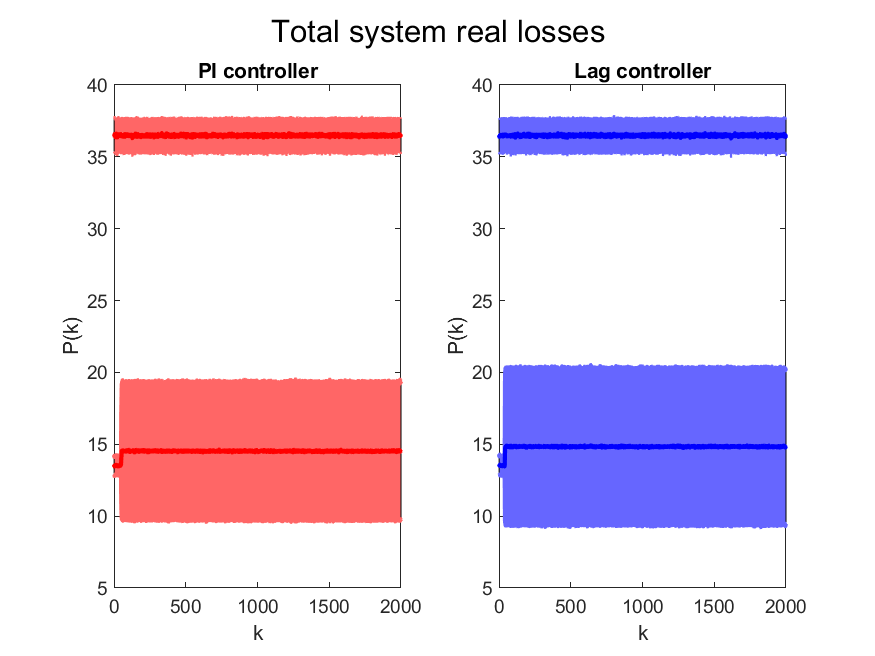}
\caption{Results of simulations on the serial test case regulated to 300 MW:
Aggregate power produced by the ensemble (left) and the corresponding lossess in transmission, as functions of time for the two controllers and two initial states of each of the two controllers.}
\label{fig:linear02p1}
\end{figure}

\begin{figure}[h]
\centering 
\includegraphics[width=0.45\textwidth]{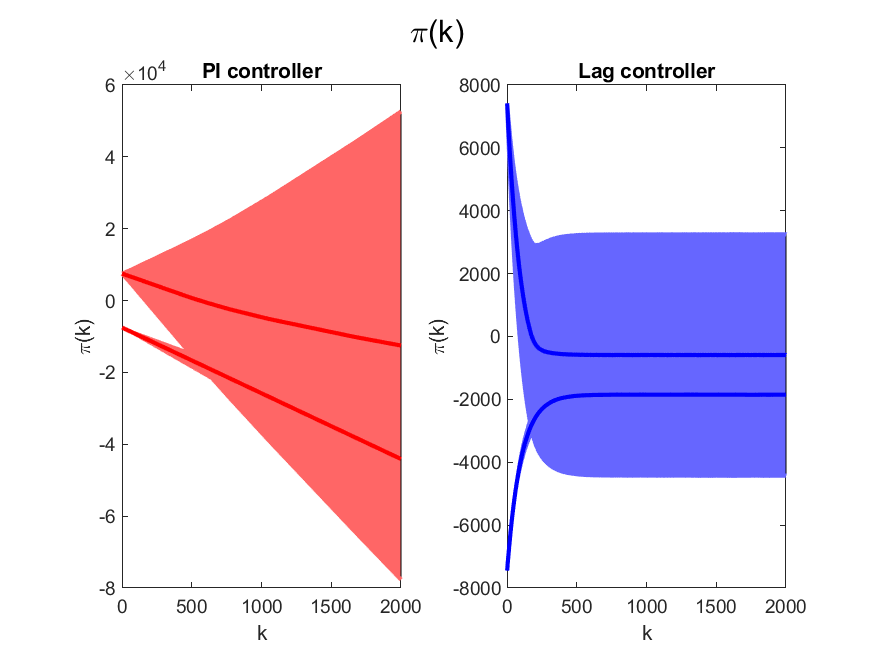}
\includegraphics[width=0.45\textwidth]{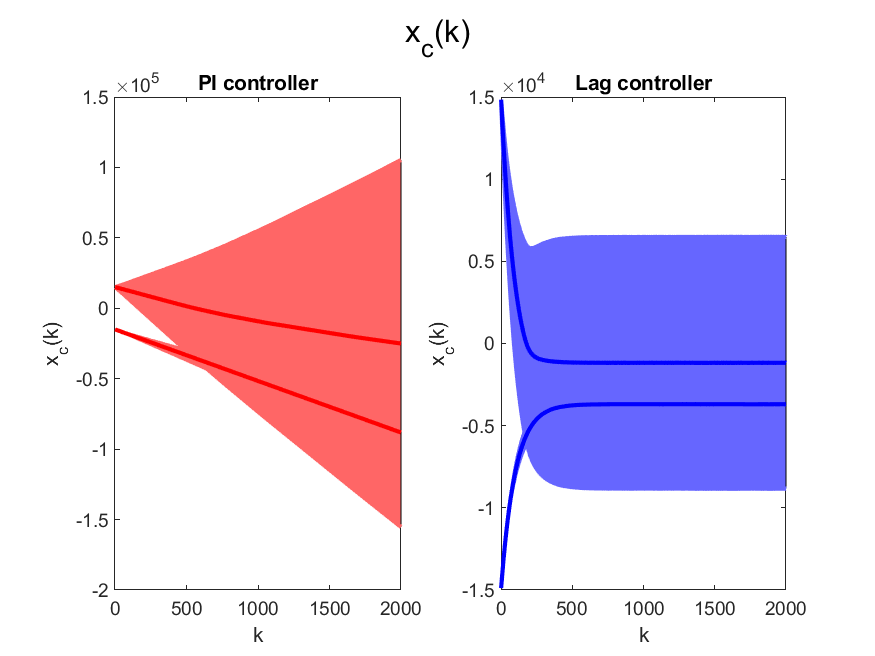}
\caption{Results of simulations on the serial test case regulated to 300 MW:
Control signal and the state of the controllers, as functions of time, for the two controllers and two initial states of each of the two controllers.}
\label{fig:linear02p2}
\end{figure}

\begin{figure*}[h]
\centering 
\includegraphics[width=0.45\textwidth]{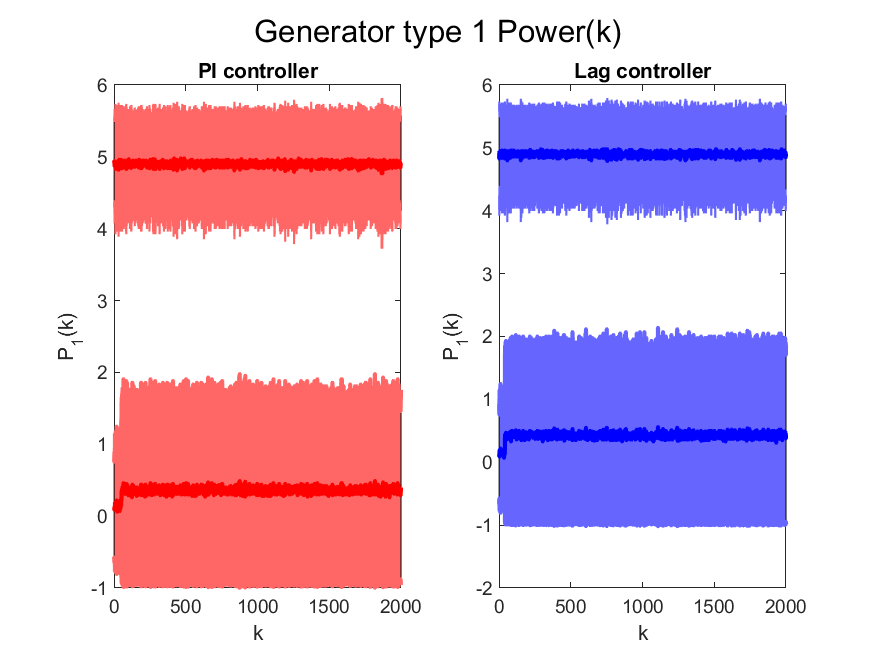}
\includegraphics[width=0.45\textwidth]{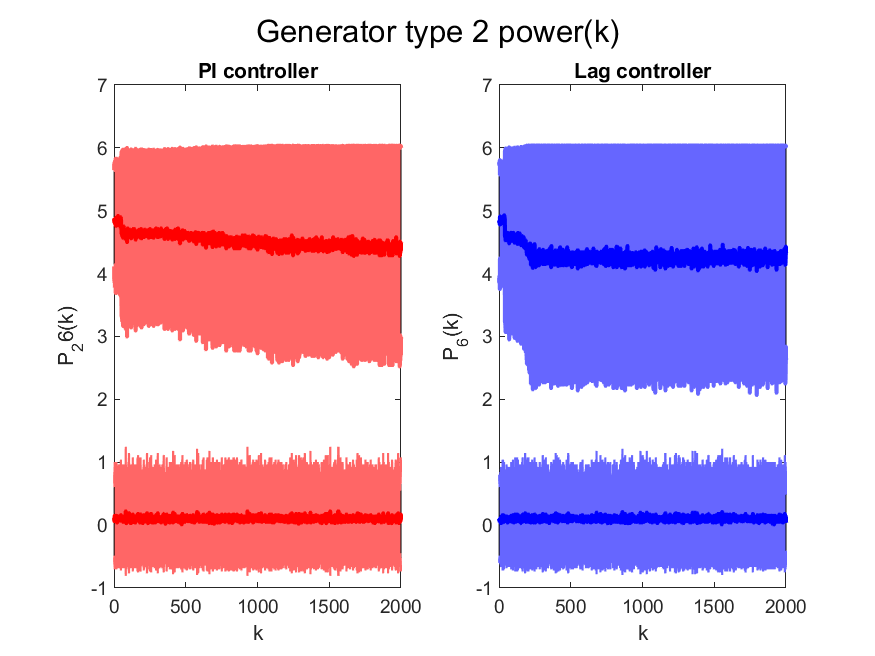}
\caption{Results of simulations on the serial test case regulated to 300 MW:
Powers at the first 60 DERs (type-1 probability functions) and following 60 DERs (type-2 probability functions).}
\label{fig:linear02p3}
\end{figure*}

In our second set of simulations on the serial test case, we repeat 500 times the simulations considering the time horizon of 2000 time steps each. Instead of $r=300$, $x_0 = 300$ of the first set of simulations, we consider
$r$ to be $300$ plus losses at time $k=0$,
and similarly $x_0$ to be $300$ plus losses at time $k=0$.
Otherwise, the system's behaviour is similar. 

As above, we are able to regulate the aggregate power output (cf. Figure ~\ref{fig:linear03p1}). With PI controller, however, the state of the controller, and consequently the signal and the power at individual buses are determined by the initial state of the controller (cf. Figures \ref{fig:linear03p2}--\ref{fig:linear03p3}).

\begin{figure}[h]
\centering 
\includegraphics[width=0.45\textwidth,trim=0 0 0 10mm,clip]{linear_02/01_Assembly_power.png}
\includegraphics[width=0.45\textwidth]{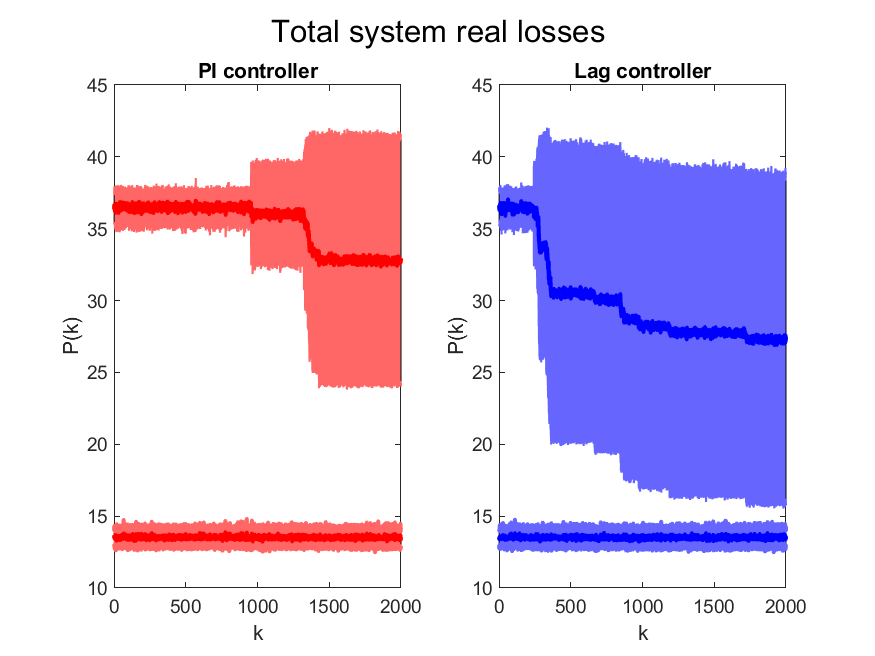}
\caption{Results of simulations on the serial test case regulated to 300 MW plus losses: 
Aggregate power produced by the ensemble (left) and the corresponding lossess in transmission, as functions of time for the two controllers and two initial states of each of the two controllers.}
\label{fig:linear03p1}
\end{figure}

\begin{figure}[h]
\centering 
\includegraphics[width=0.45\textwidth]{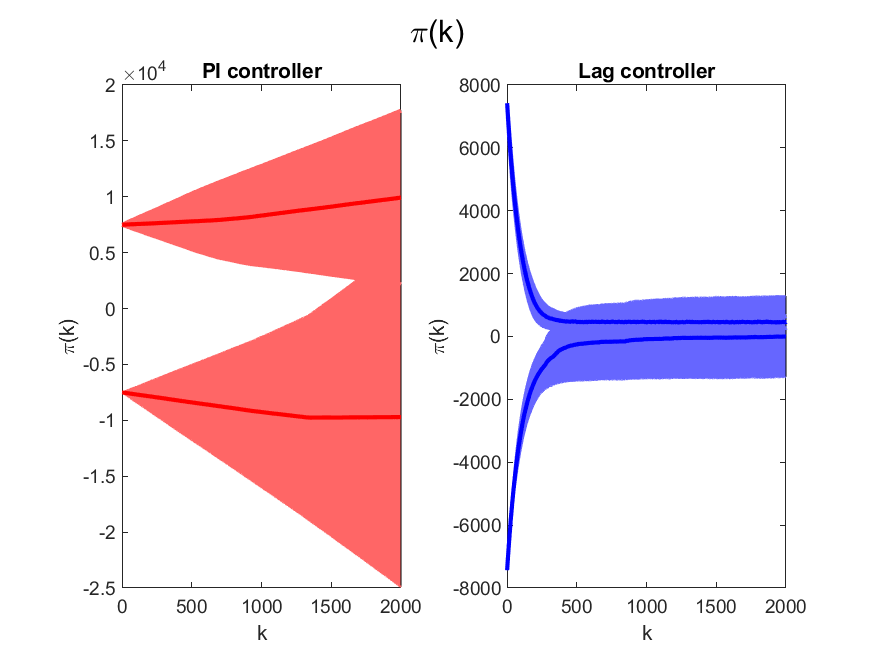}
\includegraphics[width=0.45\textwidth]{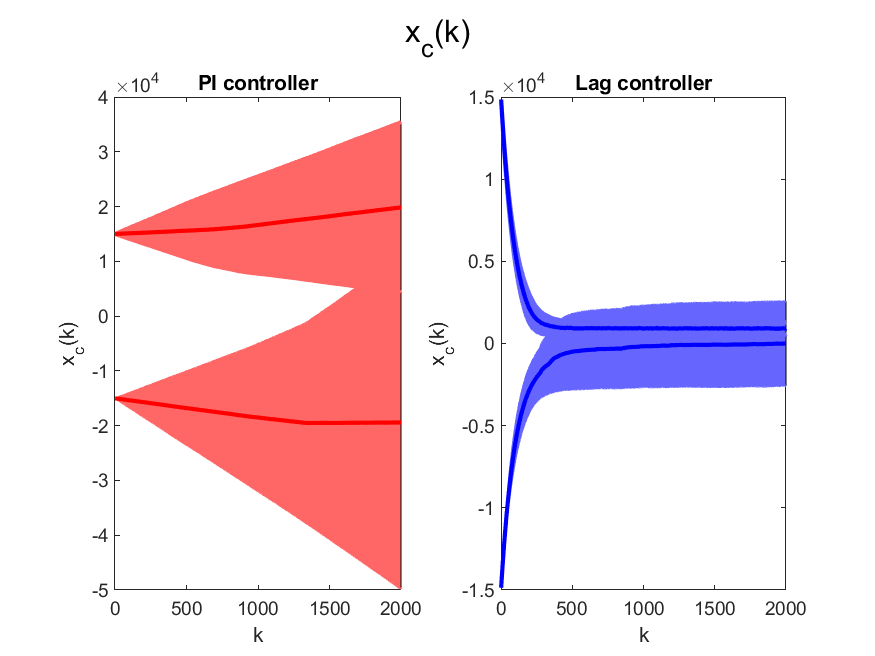}
\caption{Results of simulations on the serial test case regulated to 300 MW plus losses:
Control signal and the state of the controllers, as functions of time, for the two controllers and two initial states of each of the two controllers.}
\label{fig:linear03p2}
\end{figure}

\begin{figure*}[h]
\centering 
\includegraphics[width=0.45\textwidth]{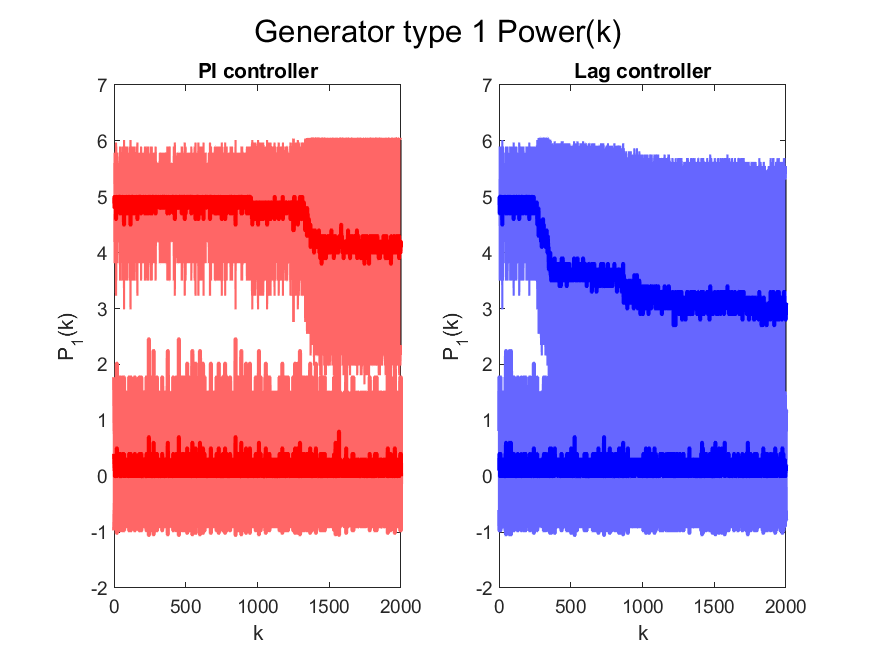}
\includegraphics[width=0.45\textwidth]{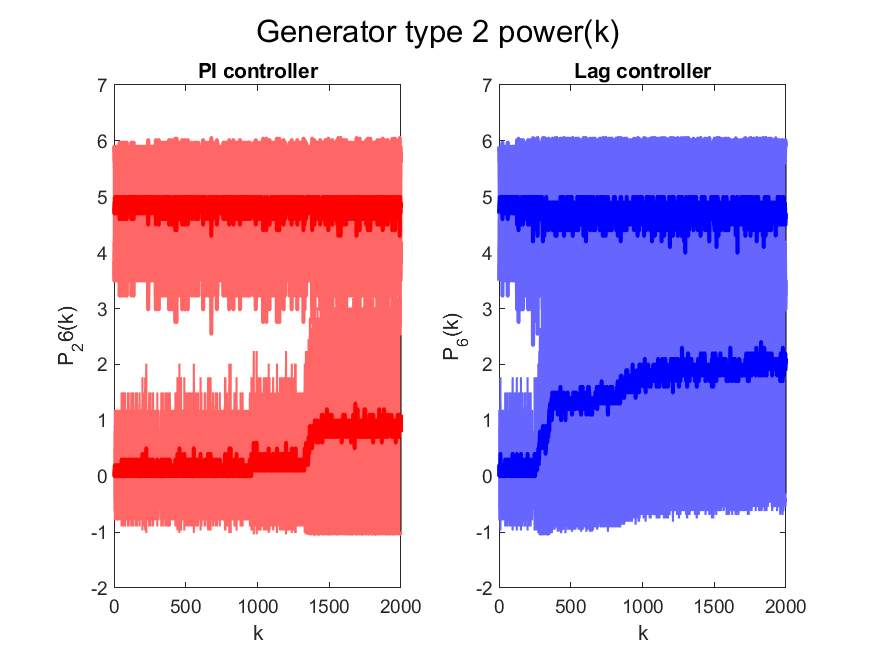}
\caption{Results of simulations on the serial test case regulated to 300 MW plus losses: 
Powers at the first 60 DERs (type-1 probability functions) and following 60 DERs (type-2 probability functions).}
\label{fig:linear03p3}
\end{figure*}

\FloatBarrier
\subsection{A Transmission-System Test Case with Losses in the Filter}

The third example demonstrates the procedure's applicability to the standard IEEE 118-bus test case with minor modifications. Notably, the ensemble is connected to two buses of the transmission system. At bus 10, 1 generator with a maximum active power output of 450 MW was replaced by 8 DERs with 110 MW of power each, out of which 4 DERs used probability function $g_{i1}$, where $x_{01}=660$.
The other four DERs used function $g_{i2}$, where $x_{02}=110$ and $\xi=100$. 
At bus 25, 1 generator with maximum active power output of 220 MW was replaced by four generators with active power output of 110 MW, out of which
two generators used probability function $f_{i1}$, and the other two used function $f_{i2}$ with $r=660$.
Bus 69 is the slack bus, similarly to the generator at bus 3 in the previous example. 
Figures \ref{fig:case118p1}--\ref{fig:case118p3} yield analogical results to the first example, for both controllers and different initial conditions $x_c (0)$.

\begin{figure}
\centering 
\includegraphics[width=0.45\textwidth]{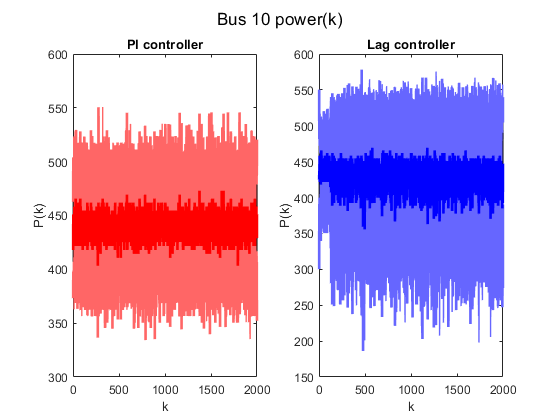}
\includegraphics[width=0.45\textwidth]{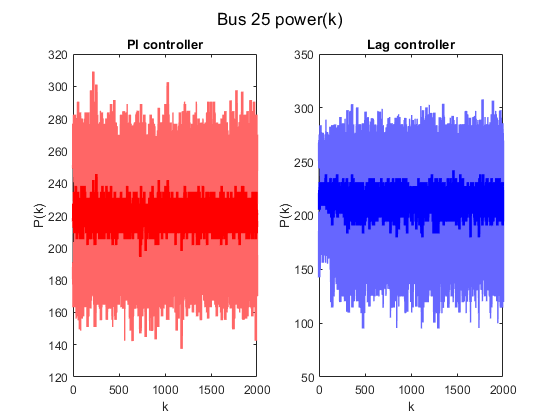}\\
\caption{Results of simulations on IEEE 118-bus test case: 
Powers at buses 10 and 25.}
\label{fig:case118p1}
\end{figure}

\begin{figure}
\centering 
\includegraphics[width=0.45\textwidth]{case118_pf_01/118_07_pi_signal.png}
\includegraphics[width=0.45\textwidth]{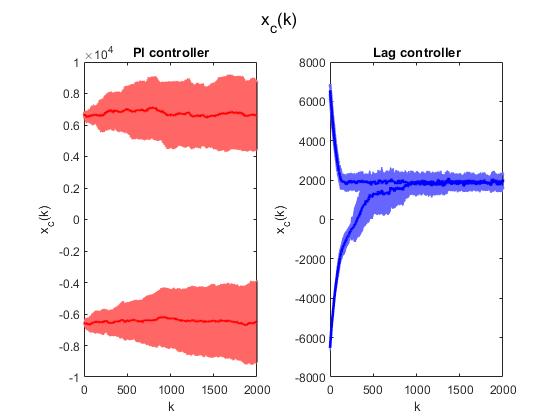}
\caption{Results of simulations on IEEE 118-bus test case: 
Control signals and the state of the controllers, as functions of time, for the two controllers and two initial states of each of those.}
\label{fig:case118p2}
\end{figure}

\begin{figure*}[h]
\centering 
\includegraphics[width=0.45\textwidth]{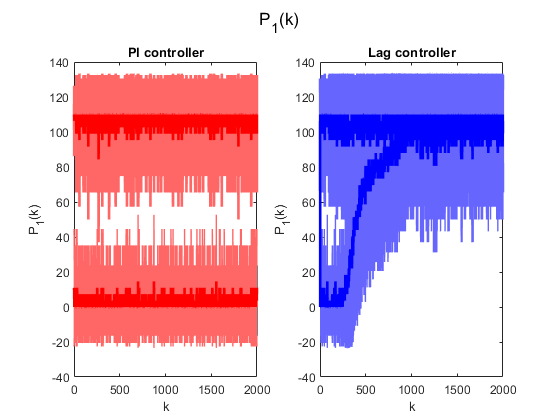}
\includegraphics[width=0.45\textwidth]{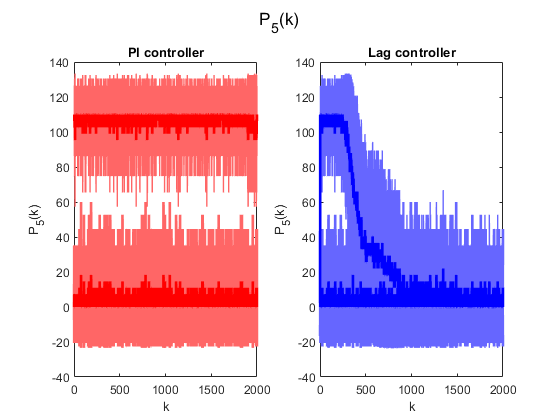}\\
\includegraphics[width=0.45\textwidth]{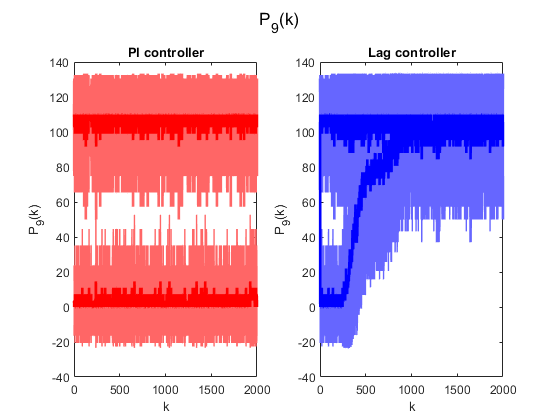}
\includegraphics[width=0.45\textwidth]{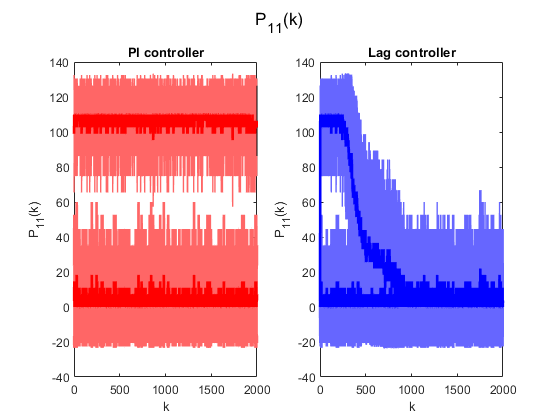}\\
\caption{Results of simulations on IEEE 118-bus test case: 
Powers at buses 1, 5, 9, and 11.}
\label{fig:case118p3}
\end{figure*}

\end{document}